\documentclass[11pt]{article}
\usepackage[english]{babel}
\usepackage{amsmath,amssymb,graphicx,enumerate,latexsym,xcolor,theorem,geometry}
\newcommand{\assign}{:=}
\newcommand{\backassign}{=:}
\newcommand{\cdummy}{\cdot}
\newcommand{\tmdummy}{$\mbox{}$}
\newcommand{\tmop}[1]{\ensuremath{\operatorname{1}}}
\newcounter{item}
\numberwithin{item}{section}
\newenvironment{enumerateroman}{\begin{enumerate}[i.] }{\end{enumerate}}
\newenvironment{proof}{\noindent\textbf{Proof\ }}{\hspace*{\fill}$\Box$\medskip}
\newtheorem{corollary}[item]{Corollary}
\newtheorem{definition}[item]{Definition}
\newtheorem{lemma}[item]{Lemma}
\newtheorem{proposition}[item]{Proposition}
{\theorembodyfont{\rmfamily}\newtheorem{remark}[item]{Remark}}
\newtheorem{theorem}[item]{Theorem}
\numberwithin{equation}{section}
\usepackage{hyperref}

\begin{document}

\title{Local and global well-posedness of 2d periodic multiplicative
stochastic NLS}

\author{Immanuel Zachhuber}

\maketitle

\begin{abstract}
  We use Strichartz estimates with rough potentials like the spatial white
  noise on the 2 \ dimensional torus to prove global well-posedness of the multiplicative stochastic NLS with general
  integer powers in both the energy and strong regime together with controls over the growth of the norms of the solutions.
\end{abstract}

\section{Introduction}\label{sec:introstrich}

This work is devoted to proving Strichartz estimates leading to low-regularity
local-in-time and high regularity global-in-time well-posedness of defocussing
NLS (nonlinear Schr{\"o}dinger equations) with very rough potentials $\xi$, so
\begin{equation}\begin{aligned}
	i \partial_t u - \Delta u & = u \cdummy \xi - u |u|^{2 n}  \text{ on }
	\mathbb{T}^2 \\
	u (0) & = u_0,  \label{eqn:intronls}
	\end{aligned}
\end{equation}
where $\mathbb{T}^2 =\mathbb{R}^2 /\mathbb{Z}^2$ is the $2$ dimensional torus
and general integer power nonlinearity $n \in \mathbb{N}$. Our chief interest
is the case where $\xi$ is \textit{spatial white noise}, which is a
distribution whose regularity is only $\mathcal{C}^{- 1 - \varepsilon}$ for
$\varepsilon > 0$, see \eqref{def:white} for its precise definition and the
appendix for a reminder of the definition of the \textit{H{\"o}lder-Besov
spaces} $\mathcal{C}^{\alpha}$.

In the case of the white noise potential there turns out to be a peculiarity
in the form of \textit{renormalisation}, which means that in order to make
sense of \eqref{eqn:intronls} one is required to shift by an infinite
correction term, formally ``$\infty \cdummy u$''. This can be interpreted as
an infinite phase shift in the PDE, since $e^{i t c} u$ solves the equation
with an additional mass $c.$ This kind of renormalisation is now well known in
the theory of singular SPDEs which has seen a rapid growth in recent years
following the introduction of the theory of \textit{Regularity Structures}
by Hairer {\cite{hairer2014theory}}, the theory of \textit{Paracontrolled
Distributions} by Gubinelli, Perkowski, and Imkeller
{\cite{gubinelli2015paracontrolled}} and others.

The approach we follow in this paper is to put the potential $\xi$ into the
definition of the operator, i.e. we try to define the operator
\[ H \text{``$=$''} \Delta + \xi \]
as a self-adjoint and semi-bounded operator on $L^2 (\mathbb{T}^d)$. This was
first done by Allez and Chouk in {\cite{allez_continuous_2015}}, where the
operator together with its domain were constructed in 2d with the white noise
potential--hereafter called the \textit{Anderson Hamiltonian}--using
Paracontrolled Distributions. A similar approach was used in {\cite{GUZ}} to
construct the operator and its domain in 3d with an eye also on solving PDEs
like \eqref{eqn:intronls}. In Section \ref{sec:anderson} we recall the main
ideas of {\cite{GUZ}} since the results are integral to the current work. The
domain of the Anderson Hamiltonian was also constructed by Labb{\'e} using
Regularity Structures {\cite{labbe2019continuous}} with Dirichlet boundary conditions and by Mouzard \cite{Mou} on compact surfaces.

The equation \eqref{eqn:intronls} with white noise potential in 2d was solved,
but not shown to be well-posed, by Debussche and Weber
{\cite{debussche2018schrodinger}} in the cubic case and by Visciglia and
Tzvetkov {\cite{TVa}},{\cite{TVb}} for other powers and on the whole space with
a sub-cubic power in the nonlinearity by Debussche and Martin in
{\cite{debussche2019solution}} which was then generalised in
{\cite{DLTVglobal}} to higher powers. In {\cite{GUZ}} global
well-posedness(GWP) in the cubic case was proved in the domain of the
Anderson Hamiltonian in 2d, whereas in 3d one gets a blow-up alternative when
starting in the domain analogously to the case of classical $\mathcal{H}^2$
solutions in {\cite{cazenave2003semilinear}}, in the more recent paper, global well-posedness was proved for a class of Hartree nonlinearities in 3d \cite{DVJZ25}. Furthermore, in {\cite{GUZ}}
global existence in the energy space in 2d was shown, but not well-posedness.
Achieving GWP for energy solutions to \eqref{eqn:intronls} is one of the
results of this paper, see Theorem \ref{thm:2denergy_intro}.

The (nonlinear) Schr{\"o}dinger equation \eqref{eqn:intronls} with a (random)
potential has certain physical interpretations, see
{\cite{gerard2006nonlinear}} and the references therein. In this paper we
consider the white noise potential but the same results hold for a large class
of potentials, see Remark \ref{rem:genpot}. Some potentials of interest are
actually \textit{critical} in the sense of scaling, like the Dirac Delta in
2d (see the monograph {\cite{albeverio2012solvable}}) or the potential $|
\cdummy |^{- 2}$ treated in {\cite{burq2003strichartz}}. Our method does not
apply in these cases, but in the aforementioned examples the analysis depends
in a crucial way on the structure of the potential.

Stochastic NLS similar to \eqref{eqn:intronls} but with different noises(e.g.
white in time coloured in space) have also been considered, see
{\cite{de1999stochastic}}, {\cite{brzezniak2019martingale}},
{\cite{fan2018global}} to name but a few. Other stochastic dispersive PDEs
which have been studied in recent years include stochastic NLS with additive
space-time noise {\cite{oh2018stochastic}}, {\cite{forlano2018stochastic}} and
nonlinear stochastic wave equations with additive space-time noise in
{\cite{gubinelli2018renormalization}} and
{\cite{gubinelli2018paracontrolled}}. Let us also mention
{\cite{gubinelli2012rough}}, where the theory of \textit{Rough Paths}--the
precursor to both Regularity Structures and Paracontrolled Distributions--is
used to solve the \textit{deterministic} low-regularity KdV equation and
which showcases nicely how tools from singular SPDEs can be applied to
non-stochastic PDE problems.\\

For the sake of completeness, we state that by local well-posedness(LWP) of the SPDE \eqref{eqn:intronls} in the space $X$ means that for every $u_0\in X$ there exists a unique solution $u\in Y\subset C([0,T];X),$ for some suitable space $Y$ with the time of existence $T=T(u_0)$ to the mild formulation of \eqref{eqn:intronls} which is 
\begin{equation}
	u (t) = e^{- i t H} u_0 - i \int^t_0 e^{- i (t - s) H} u | u |^{2 n} (s) d
	s.\label{eqn:intromild}
\end{equation}
and that the map $u_0\mapsto u$ is continuous.\\
We say that the equation \eqref{eqn:intronls} is globally well-posedness(GWP) if for any time $T>0$ (i.e. independently of the initial data) there exists a unique solution to \eqref{eqn:intromild} which depends continuously on the data. These properties are usually obtained by a contraction argument, however we will see that in the energy space this is not directly possible.
We state the main (shortened) results of the paper relating to the
multiplicative stochastic NLS. $H \text{``$=$''} \Delta + \xi - \infty$ is the
Anderson Hamiltonian whose exact definition and properties are recalled in
Section \ref{sec:anderson}.\\
\begin{theorem}
  \label{thm:2dandstr(intro)}[2d Anderson Strichartz Estimates]Let $r
  \geqslant 4$, then we have for any $\delta > 0$
  \begin{equation}
    \|e^{- itH} u\|_{L_{t ; [0, 1]}^r L^r} \lesssim
    \|u\|_{\mathcal{H}^{\frac{(r - 3) (1 + \delta)}{r}}} \label{eqn:introstr}
  \end{equation}
\end{theorem}

\begin{theorem}
  \label{thm:2dlwp_intro}[2d low regularity local well-posedness] Let $n \in
  \mathbb{N},$ then the SPDE
  \begin{align*}
       (i \partial_t - H) u & =  - u |u|^{2 n}  \text{ on } \mathbb{T}^2\\
       u (0) & =  u_0
     \end{align*} 
  is locally well-posed (LWP) in $\mathcal{H}^s$ for $s \in \left( 1 -
  \frac{1}{2n}, 1 \right)$ up to a time $T \sim (1 + \| u_0
  \|_{\mathcal{H}^s})^{- K}$for some $K > 0$ depending on $n$
  polynomially.
\end{theorem}

\begin{theorem}
  \label{thm:2denergy_intro} [2d GWP for energy solutions] Let $n \in
  \mathbb{N},$ then the SPDE
   \begin{align*}
       (i \partial_t - H) u & = - u |u|^{2 n}  \text{on } \mathbb{T}^2\\
       u (0) & = u_0
     \end{align*} 
  is globally well-posed (GWP) in the energy space, $\mathcal{D} \left(
  \sqrt{- H} \right)$, whose definition is recalled in Theorem
  \ref{thm:Gammaemb}.
\end{theorem}

\begin{theorem}
  \label{thm:stronggwp_intro}[2d GWP for strong solutions] Let $n \in
  \mathbb{N},$ then PDE
  \begin{align*}
       (i \partial_t - H) u & =  - u |u|^{2 n}  \text{on } \mathbb{T}^2\\
       u (0) & =  u_0
     \end{align*} 
  is globally well-posed(GWP) in the domain $\mathcal{D} (H)$ of the operator
  $H$, see Section \ref{sec:anderson}. Moreover we have that the norm grows at most polynomially in time, see Theorem \ref{thm:normgrowth} for a precise statement.
\end{theorem}

\begin{remark}
  \label{rem:comparison}In a previous version of the paper, we claimed that
  the bound \eqref{eqn:introstr} holds with only arbitrarily small $\delta >
  0$ loss in regularity but the proof contained an error. In the meantime,
  together with Mouzard {\cite{MZ}}, we proved a version of the theorem on
  smooth surfaces, here we now just give a simpler proof which uses the
  heavier machinery of Bourgain's periodic Strichartz
  estimates{\cite{boustr}}, {\cite{bourgain2015proof}}. In the current setting, where the potential is two-dimensional white noise, this proof doesn't improve on the estimate on a generic compact manifold, however in the three dimensional analogue, i.e. when the potential is white noise on the torus $\mathbb{T}^3$, this approach yields a useable Strichartz estimate as we show in \cite{DVJZ25}, whereas repeating the proof from \cite{MZ} does not yield a useable estimate.

  It is still an open interesting question whether the loss in derivatives
  can be reduced on the torus or manifolds with special geometries. For
  general compact surfaces, the results from {\cite{MZ}} can likely not be
  improved since it is as good as {\cite{BGT}}, which is optimal for general
  surfaces, with an $\varepsilon$ loss.
\end{remark}

The paper is organised as follows: In Section \ref{sec:strichartz} we recall
the well-known Strichartz estimates on the whole space and how their
counterparts on the torus differ. Section \ref{sec:anderson} is meant to
recapitulate the construction of the Anderson Hamiltonian and its domain
following {\cite{GUZ}}. In Section \ref{sec:strand} we prove the Strichartz
estimates for the Anderson Hamiltonian on $\mathbb{T}^2$ i.e. Theorem
\ref{thm:2dandstr(intro)}. Then in Section \ref{sec:solving} we utilise these
bounds to prove well-posedness of the multiplicative stochastic NLS in three
different regimes i.e. Theorems \ref{thm:2dlwp_intro},
\ref{thm:2denergy_intro} and \ref{thm:stronggwp_intro}.

{\textbf{Notations and conventions}}\\
The spaces we work in are $L^p$-spaces, for $p \in [1, \infty]$, meaning the
usual $p$-integrable Lebesgue functions; $\mathcal{H}^{\alpha}, W^{\alpha, p}$
spaces, with $\alpha \in \mathbb{R},\ p \in [1, \infty]$ the usual Sobolev
potential spaces with $\mathcal{H}^{\alpha} = W^{\alpha, 2} = B_{2,
2}^{\alpha}$; and $B_{r, q}^s$, the Besov spaces, whose definition is recalled
in the appendix and which cover $\mathcal{H}^{\alpha}$ and
$\mathcal{C}^{\alpha}$--so called H{\"o}lder-Besov spaces--as special cases.

Also we write
\[ \|f\|_X \assign \|f\|_{X (\mathbb{T}^2)} \text{ and } \|f (t)\|_{Y_{t ; [0,
   T]}} \assign \|f\|_{Y ([0, T])}, \]
where $X$ is one of the function spaces above on the torus $\mathbb{T}^2$ $Y$
is a function space in the time variable, usually $C [0, T], L^p [0, T]$ for
$1 \leqslant p \leqslant \infty$ and $T > 0.$

We write, as is quite common,
\[ a \lesssim b \]
to mean $a \leqslant Cb$ for a constant $C > 0$ independent of $a, b$ and
their arguments. Also we write
\[ a \sim b \Leftrightarrow a \lesssim b \text{ and } b \lesssim a. \]
For the sake of brevity we also allow {\textbf{every}} constant to depend
exponentially on the relevant noise norm $\| \Xi \|_{\mathcal{X}^{\alpha}}$,
see Definition \ref{def:2dnoise} for the exact definition of the norms; This
can be written schematically as
\[ \lesssim \Leftrightarrow \lesssim_{\Xi}, \]
this comes with the tacit understanding that everything is continuous with
respect to this norm. Another convention is that if we write something like
\[ \|F (u)\|_X \lesssim \|u\|_{\mathcal{H}^{\alpha + \varepsilon}}  \text{for
   } \varepsilon > 0, \]
we of course mean
\[ \|F (u)\|_X \leqslant C_{\varepsilon} \|u\|_{\mathcal{H}^{\alpha +
   \varepsilon}}  \text{ with } C_{\varepsilon} \rightarrow \infty \text{ as }
   \varepsilon \rightarrow 0. \]
\textbf{Acknowledgments.} This is a major reworking of a paper that
appeared as a part of the author's PhD dissertation \cite{diss}  which was inspired by fruitful
discussions with Professors Massimiliano Gubinelli and Herbert Koch and was partially supported by the German Research Foundation (DFG) via CRC 1060.\\
The author acknowledges funding by the Deutsche Forschungsgemeinschaft (DFG, German Research Foundation) – CRC/TRR 388 "Rough Analysis, Stochastic Dynamics and Related Fields“ – Project ID 516748464.
\\


\section{Classical Strichartz estimates on the torus}\label{sec:strichartz}

We start by recalling the well-known Strichartz estimates for Schr{\"o}dinger
equations on $\mathbb{R}^d$.

\begin{theorem}
  \label{thm:strwhole}[Strichartz on $\mathbb{R}^d$,Theorem 2.3 in {\cite{tao2006nonlinear}}]Let $d \geqslant
  1$ and $(p, q)$ be a Strichartz pair, i.e.
  \[ \frac{2}{p} + \frac{d}{q} = \frac{d}{2}  \text{and } (d, p, q) \neq (2,
     2, \infty), \]
  we also take $(r', s')$ to be a dual Strichartz pair, which means that they
  are H{\"o}lder duals of a Strichartz pair $(r, s)$, explicitly
  \[ \frac{2}{r'} + \frac{d}{s'} = \frac{d + 4}{2}, \]
  then the following are true
  \begin{enumerateroman}
    \item $\|e^{it \Delta} u\|_{L_{t ; \mathbb{R}}^p L^q
    (\mathbb{R}^d)} \lesssim \|u\|_{L^2 (\mathbb{R}^d)}$ ``homogeneous
    Strichartz estimate''
    
    \item $\left| \int_{\mathbb{R}} e^{- it \Delta} F (t) dt \right|_{L^2
    (\mathbb{R}^d)} \lesssim \|F\|_{L^{r'} (\mathbb{R}) L^{s'}
    (\mathbb{R}^d)}$ ``dual homogeneous Strichartz estimate''
    
    \item $\left| \int_{t' < t} e^{i (t - t') \Delta} F (t') dt'
    \right|_{L_t^p (\mathbb{R}) L^q (\mathbb{R}^d)} \lesssim \|F\|_{L^{r'}
    (\mathbb{R}) L^{s'} (\mathbb{R}^d)}$ ``inhomogeneous Strichartz
    estimates''.
  \end{enumerateroman}
\end{theorem}

Next we cite some, by now, classical Strichartz estimates on the torus and how they
differ from those on the whole space. Moreover we sketch how they allow to
solve NLS in spaces below $\mathcal{H}^{\frac{d}{2} + \varepsilon}$,which is
an algebra.

The first results we state are the Strichartz estimates proved by
Burq-Gerard-Tzvetkov in {\cite{burq2004strichartz}}. They hold on general
manifolds, i.e. not only the torus. The results are not optimal for the torus
but we nonetheless cite them because the methods we use are strongly inspired
by this paper.

\begin{theorem}
  \label{thm:classStr}[Strichartz estimates on compact
  manifolds,{\cite{burq2004strichartz}} Theorem 1] Let
  \[ \frac{2}{p} + \frac{d}{q} = \frac{d}{2} . \]
  We have on the finite time interval $[0, 1]$
  \begin{equation}
    \|e^{- it \Delta} u\|_{L_{t ; [0, 1]}^p L^q(M)} \lesssim
    \|u\|_{\mathcal{H}^{\frac{1}{p}}(M)} \label{eqn:classStr}
  \end{equation}
  and
  \[ \left| \int^t_0 e^{- i (t - s) \Delta} f (s) ds \right|_{L_{t ; [0, 1]}^p
     L^q(M)} \lesssim \int^1_0 \|f (s)\|_{\mathcal{H}^{\frac{1}{p}}(M)} ds. \]
\end{theorem}

Note that, as opposed to the whole space, one has a loss of $\frac{1}{p}$
derivatives. Together with Mouzard, we proved an analogous Strichartz estimate
on smooth surfaces using a microlocal approach as in {\cite{BGT}}, see also
{\cite{Mou}} for the construction of the operator $H$ on a smooth surface.

\begin{theorem}
  \label{thm:MZStr}[Anderson Strichartz estimates on smooth surfaces,
  {\cite{MZ}}] Let $M$ be a smooth compact surface and
  \[ \frac{2}{p} + \frac{d}{q} = 1. \]
  We have on the finite time interval $[0, 1]$
  \[ \|e^{- itH} u\|_{L_{t ; [0, 1]}^p L^q} \lesssim
     \|u\|_{\mathcal{H}^{\frac{1}{p} + \varepsilon}} \]
  and
  \[ \left| \int^t_0 e^{- i (t - s) H} f (s) ds \right|_{L_{t ; [0, 1]}^p L^q}
     \lesssim \int^1_0 \|f (s)\|_{\mathcal{H}^{\frac{1}{p} + \varepsilon}} ds
  \]
  for any $\varepsilon > 0.$
\end{theorem}

The next result we cite is an almost sharp Strichartz estimate on the torus
due to Bourgain and Demeter in {\cite{bourgain2015proof}} which was refined in
{\cite{killip2016scale}} by Killip and Visan whose version we cite because it
is more amenable to our situation. The result is stated for functions which
are localised in frequency but the corresponding Sobolev bound is immediate.

\begin{theorem}
  \label{thm:imprstr}[Sharp Strichartz estimate, Theorem 1.2
  {\cite{killip2016scale}}, {\cite{bourgain2015proof}}]Let $d \geqslant 1$ and
  $p \geqslant \frac{2 (d + 2)}{d}$, then, for any $\varepsilon > 0$ we have
  \[ \|e^{- it \Delta} P_{\leqslant N} f\|_{L_{t ; [0, 1]}^p L^p
     (\mathbb{T}^d)} \lesssim N^{\frac{d}{2} - \frac{d + 2}{p} + \varepsilon}
     \|f\|_{L^2 (\mathbb{T}^d)} . \]
  For $d = 2$ this means $p \geqslant 4$.
\end{theorem}

This will allow us to give a comparatively simple proof of Theorem
\ref{thm:2dandstr(intro)} which is very similar to Theorem \ref{thm:MZStr}. In
order to prove this, we first state the following simple result.

\begin{proposition}
  \label{cor:scaling}[Inhomogeneous short time bound]
  
  Let $I = [t_0, t_1]$ be a subinterval of $[0, 1]$ and $4 \leqslant p <
  \infty$ and let $f \in L_{[0, 1]}^{\infty} {\mathcal{H}^{1 - \frac{4}{p} +
  \varepsilon}} $ for $\varepsilon > 0$, then
  \begin{eqnarray}
    \left\| \int^t_{t_0} e^{- i (t - s + t_0) \Delta} f (s) d s
    \right\|_{L_I^p L^p} & \lesssim \int_I \|f (s) \|_{\mathcal{H}^{1 -
    \frac{4}{p} + \varepsilon}} d s \\
    & \lesssim | I | \|f\|_{L_I^{\infty} \mathcal{H}^{1 - \frac{4}{p} +
    \varepsilon}} . 
  \end{eqnarray}
\end{proposition}

\begin{proof}
  The bound follows in the usual way that one obtains inhomogeneous estimtes
  from the linear Strichartz estimate Theorem \ref{thm:imprstr}.
\end{proof}

Furthermore, we give a quick sketch about why these kinds of estimates are
used for solving NLS.

Take for simplicity the cubic NLS on the two-dimensional torus.
\begin{eqnarray*}
  i \partial_t u - \Delta u & = & - u |u|^2  \text{on } \mathbb{T}^2\\
  u (0) & = & u_0 .
\end{eqnarray*}
The Duhamel formula reads
\begin{equation}
  u (t) = e^{- it \Delta} u_0 - i \int^t_0 e^{- i (t - s) \Delta} |u|^2 u (s)
  ds. \label{eqn:2dnls-class}
\end{equation}
Since for $u_0 \in \mathcal{H}^{\sigma}$ with $\sigma \in \mathbb{R}$ we have
\[ e^{- it \Delta} u_0 \in C_t \mathcal{H}^{\sigma} \]
with the continuity in time following from Stone's theorem, it is natural to
try to solve \eqref{eqn:2dnls-class} in a space like $C_t
\mathcal{H}^{\sigma}$ for, say, $\sigma \geqslant 0$. Now, since the unitary
group $e^{- it \Delta}$ has no smoothing properties, one possible way
to bound the nonlinear expression in \eqref{eqn:2dnls-class} is

\begin{align}
  \left| \int^t_0 e^{- i (t - s) \Delta} |u|^2 u (s) ds \right|_{C_{t ; [0,
  T]} \mathcal{H}^{\sigma}} & \lesssim \int^T_0 \||u|^2 u
  (s)\|_{\mathcal{H}^{\sigma}} ds \\
  & \lesssim \int^T_0 \|u (s)\|^2_{L^{\infty}} \|u
  (s)\|_{\mathcal{H}^{\sigma}} ds,  \label{eqn:CHbound}
\end{align}

where the second inequality follows from the ``tame'' estimate (see Lemma
\ref{lem:tame}). In the case that $\sigma > \frac{d}{2}$($= 1$ in 2d) the
$L^{\infty}$ norm is controlled by the $\mathcal{H}^{\sigma}$ norm so it is
easy to close the fixed point argument. But even in the case of
$\mathcal{H}^1$, which is natural as it is the ``energy space'', one is not
able to close the contraction argument without additional input.

The key observation to make -- and to see where the Strichartz estimates enter
-- is that in \eqref{eqn:CHbound} we can apply H{\"o}lder's inequality in time
to obtain
\[ \int^T_0 \|u (s)\|^2_{L^{\infty}} \|u (s)\|_{\mathcal{H}^{\sigma}} ds
   \lesssim \|u\|^2_{L_{[0, T]}^2 L^{\infty}} \|u\|_{L_{[0, T]}^{\infty}
   \mathcal{H}^{\sigma}} \]
and note that we need not control the $L_t^{\infty} L_x^{\infty}$ norm of the
solution but it suffices to control the $L_t^2 L_x^{\infty}$ norm.

So, if we were able to control the $L_t^p L_x^{\infty}$ norm of the right-hand
side of \ \eqref{eqn:2dnls-class} by the $C_t \mathcal{H}^{\sigma}$ norm of
$u$ we would be able to get a contraction to get a solution to
\eqref{eqn:2dnls-class} for a short time interval. Bounding the
$L^{\infty}$ norm directly is hopeless but recall the Sobolev embedding in $d
-$dimensions
\[ W^{\frac{2}{q} + \varepsilon, q} \hookrightarrow L^{\infty}  \text{ for any
   } q \in (1, \infty)  \text{ and } \varepsilon > 0. \]
This is the point where the Strichartz estimates come in, since, for example
by Theorem \ref{thm:imprstr}, one gets the bound
\[ \|e^{- it \Delta} u_0 \|_{L_{t ; [0, T]}^p L^{\infty}} \lesssim \|e^{- it
   \Delta} u_0 \|_{L_{t ; [0, T]}^p W^{\frac{2}{p} + \varepsilon, p}}
   \lesssim_T \|u_0 \|_{\mathcal{H}^{1 - \frac{2}{p} + \varepsilon}} \quad
   \text{for } 4 \leqslant p < \infty \]
for the linear evolution. If we take $\varepsilon > 0$ small then the
regularity exponent $1 - \frac{2}{p} + \varepsilon < 1$ so this is strictly
better than what we would get from estimating the $L^{\infty}$ norm by the
$\mathcal{H}^{1 + \varepsilon}$ norm.

For the nonlinear term we get similarly (assuming for simplicity $T \le 1$)

\begin{align*}
  \left\| \int^t_0 e^{- i (t - s) \Delta} |u|^2 u (s) ds \right\|_{L_{t ; [0,
  T]}^p L^{\infty}} & \lesssim \left\| \int^t_0 e^{- i (t - s) \Delta} |u|^2 u
  (s) ds \right\|_{L_{t ; [0, T]}^p W^{\frac{2}{p} + \varepsilon, p}}\\
  & \lesssim \int^T_0 \||u|^2 u (s)\| ds_{\mathcal{H}^{1 - \frac{2}{p} +
  \varepsilon}}\\
  & \lesssim T \|u\|^3_{L_{[0, T]}^{\infty} \mathcal{H}^{\tilde{\sigma}}}
\end{align*}

where $1 - \frac{2}{p} + \varepsilon < \tilde{\sigma} < 1$ can be computed
explicitly using the fractional Leibniz rule, see Lemma \ref{lem:fracleib}.

Clearly these bounds can be sharpened in different ways but the important
thing is that Strichartz estimates lead to local-in-time well-posedess for
some range of $\sigma \leqslant 1$.

\section{The Anderson Hamiltonian in 2 dimensions}\label{sec:anderson}

One aim of this work is to establish Strichartz estimates for the
Anderson Hamiltonian which is formally given by
\[ H \text{``$=$''} \Delta + \xi (x) - \infty, \]
where $\xi (x)$ is spatial white noise, see \eqref{def:white}. This operator
was initially studied by Allez and Chouk on $\mathbb{T}^2$ in
{\cite{allez_continuous_2015}} and later by Gubinelli,Ugurcan and the author
in {\cite{GUZ}} on $\mathbb{T}^2$ and $\mathbb{T}^3$ using the theory of
\textit{Paracontrolled Distributions} which was introduced in
{\cite{gubinelli2015paracontrolled}}. The operator was also studied by
Labb{\'e} in {\cite{labbe2019continuous}} using the theory of
\textit{Regularity Structures} introduced in {\cite{hairer2014theory}} also
including boundary conditions and on surfaces by Mouzard in {\cite{Mou}}.
Naively one might think that it is simply a suitably well-behaved perturbation
of the Laplacian in which case Theorem 6 in {\cite{burq2004strichartz}} would
more or less directly apply. However, it was shown that the domain of $H$ in
both $2 d$ and $3 d$ can be quite explicitly determined and we even have (almost surely)
\[ \mathcal{D} (H) \cap \mathcal{H}^2 = \{ 0 \}, \]
so it is tricky to directly compare the operators $H$ and $\Delta .$

\

\

We briefly recall some of the main ideas from {\cite{GUZ}} in the 2d setting
and slightly reformulate it. An observation made in
{\cite{allez_continuous_2015}} was that a function $u$ is in the domain of $H$
if
\begin{equation}
  u - (u \prec (1 - \Delta)^{- 1} \xi + B_{\Xi} (u)) \in \mathcal{H}^2,
  \label{eqn:ACdomain}
\end{equation}
see the appendix for the definition and properties of paraproducts. By the
paraproduct estimates(Lemma \ref{lem:para}) and the regularity of the noise,
the term $u \prec (1 - \Delta)^{- 1} \xi$ is no better than $\mathcal{H}^{1 -
\varepsilon} .$ The ``lower order'' correction term $B_{\Xi}$ is also worse
than $\mathcal{H}^2$ (in fact it is $\mathcal{H}^{2 - \varepsilon}$). This for
example rules out that $u$ is regular, rather it fixes its regularity at
$\mathcal{H}^{1 - \varepsilon} .$ See Definition \ref{def:2dnoise} for the
exact definition of the \textit{enhanced noise} $\Xi .$

One of the chief innovations in {\cite{GUZ}} as opposed to
{\cite{allez_continuous_2015}} was that by observing that the statement
\eqref{eqn:ACdomain} is equivalent to
\[ u - P_{> N} (u \prec (1 - \Delta)^{- 1} \xi + B_{\Xi} (u)) \in
   \mathcal{H}^2, \]
where $P_{> N} =\mathcal{F}^{- 1} \mathbb{I}_{> N} \mathcal{F},$ for any $N >
0$ cuts out the low frequencies. By choosing $N$ large enough depending on the
$\mathcal{X}^{\alpha}$ norm of $\Xi$ (see Definition \ref{def:2dnoise}) it was
shown that the map
\begin{eqnarray*}
  \Phi (u) & \assign & u - P_{> N (\Xi)} (u \prec (1 - \Delta)^{- 1} \xi +
  B_{\Xi} (u))\\
  \mathcal{D} (H) & \mapsto & \mathcal{H}^2
\end{eqnarray*}
which sends a paracontrolled function to its remainder admits an inverse which
we call $\Gamma$ and we rename $\Phi$ as $\Gamma^{- 1} .$

\

In the following we use the short-hand notation
\begin{equation}
  u = \Gamma u^{\sharp} = P_{> N (\Xi)} (\Gamma u^{\sharp} \prec (1 -
  \Delta)^{- 1} \xi + B_{\Xi} (\Gamma u^{\sharp})) + u^{\sharp},
  \label{eqn:Gamma2d}
\end{equation}
where the term $B_{\Xi}$ is explicitly given by
\[ B_{\Xi} (u) \assign (1 - \Delta)^{- 1} (\Delta u \prec X + 2 \nabla u \prec
   \nabla X + \xi \prec u - u \prec (\xi \diamondsuit X)), \]
with
\[ X = (1 - \Delta)^{- 1} \xi \text{ and } \xi \diamondsuit X = \Xi_2 
   \text{ is the second component of } \Xi, \text{ see }  \eqref{defXi2} . \]
Moreover, in the new coordinates, $u^{\sharp}$, the operator $H$ is given by
\begin{align}
  H \Gamma u^{\sharp} =& (\Delta-1) u^{\sharp} + u^{\sharp} \circ \xi +
  P_{\leqslant N} (\Gamma u^{\sharp} \prec \xi + \Gamma u^{\sharp} \succ \xi)
  + P_{> N} (- B_{\Xi} (\Gamma u^{\sharp}) - \Gamma u^{\sharp} \prec X +\nonumber\\
  +&\Gamma u^{\sharp} \succeq \Xi_2 + C (\Gamma u^{\sharp}, X, \xi) + B_{\Xi}
  (\Gamma u^{\sharp}) \circ \xi), \label{eqn:defHGamma2d}
\end{align}
recalling the convention \textquotedblleft $\succeq=\succ+\circ$\textquotedblright
It is also natural to consider the operator $H$ conjugated by $\Gamma$, i.e.
\begin{equation}
  H^{\sharp} \assign \Gamma^{- 1} H \Gamma, \label{2dgammaH}
\end{equation}
which can be expressed as
\begin{align}
  H^{\sharp} u^{\sharp} & = H \Gamma u^{\sharp} - P_{> N} (H \Gamma
  u^{\sharp} \prec X + B_{\Xi} (H \Gamma u^{\sharp})) \nonumber\\
  & =  (\Delta - 1) u^{\sharp} + u^{\sharp} \circ \xi + P_{\leqslant N}
  (\Gamma u^{\sharp} \prec \xi + \Gamma u^{\sharp} \succ \xi) + P_{> N} (-
  B_{\Xi} (\Gamma u^{\sharp}) - \Gamma u^{\sharp} \prec X +\nonumber\\& +\Gamma u^{\sharp}
  \succeq \Xi_2 + C (\Gamma u^{\sharp}, X, \xi) + B_{\Xi} (\Gamma u^{\sharp})
  \circ \xi) - P_{> N} (H \Gamma u^{\sharp} \prec X + B_{\Xi} (H \Gamma
  u^{\sharp})) .  \label{eqn:2dHsharp}
\end{align}
We remark that while $H$ was shown to be self-adjoint on $L^2$, the \textquotedblleft sharpened\textquotedblright
operator $H^{\sharp}$ is not and in particular the map $\Gamma$ is not
unitary.

We quote some results from {\cite{GUZ}} and {\cite{MZ}}, this result assumes
that we have shifted the operator by a constant. We tacitly assume $- H$ to be
positive as opposed to just being semi-bounded which is achieved by adding a
constant depending on the noise as in {\cite{GUZ}}.

\begin{theorem}
  \label{thm:Gammaemb}[Proposition 2.27, Lemma 2.33, Lemma 2.34
  {\cite{GUZ}},Proposition 1.14 {\cite{MZ}}]We have, writing again $u = \Gamma
  u^{\sharp},$
  \begin{enumerateroman}
    \item $\| u \|_{\mathcal{D} (H)} = \| H u \|_{L^2} \sim \| u^{\sharp}
    \|_{\mathcal{H}^2}$
    
    \item $\| u \|_{\mathcal{D} \left( \sqrt{- H} \right)} = \left\| \sqrt{-
    H} u \right\|_{L^2} = (- (u, H u)_{L^2})^{\frac{1}{2}} \sim \| u^{\sharp}
    \|_{\mathcal{H}^1}$
    
    \item $D (H) \hookrightarrow L^{\infty}$ \text{and} $D \left( \sqrt{- H}
    \right) \hookrightarrow L^p$ \text{for any} $2 \leqslant p < \infty$
    
    \item  $\| u^{\sharp} \|_{\mathcal{H}^s} \sim \| u \|_{\mathcal{H}^s} \sim
    \left\| H^{\frac{s}{2}} u \right\|_{L^2}$ for $s \in (- 1, 1)$.
    
    \ 
  \end{enumerateroman}
\end{theorem}

The following proposition quantifies the idea that the transformed operator
$H^{\sharp} \assign \Gamma^{- 1} H \Gamma$ is a lower-order perturbation of
the Laplacian.

\begin{proposition}
  \label{prop:str:2dpert}Take $u^{\sharp} \in \mathcal{H}^2,$ then the
  following holds for any $\varepsilon, \kappa > 0$ with $1 + \varepsilon +
  \kappa \leqslant 2$
  \[ \| (H^{\sharp} - \Delta) u^{\sharp} \|_{\mathcal{H}^{\kappa}} \lesssim \|
     u^{\sharp} \|_{\mathcal{H}^{1 + \varepsilon + \kappa}} \]
    \end{proposition}
  \begin{proof}
    This essentially follows by noting that in terms of regularity the worst
    term to bound in \eqref{eqn:2dHsharp} is $u^{\sharp} \circ \xi$ which is
    bounded by(see Lemma \ref{lem:para})
    \[ \| u^{\sharp} \circ \xi \|_{\mathcal{H}^{\kappa}} \lesssim \|
       u^{\sharp} \|_{\mathcal{H}^{1 + \varepsilon + \kappa}} \| \xi
       \|_{\mathcal{C}^{- 1 - \varepsilon}} . \]
    The other terms are bounded similarly by $\mathcal{H}^s$ norms of
    $u^{\sharp}$ with $s < 1$ multiplied by H{\"o}lder norms of objects
    related to $\xi$ which appear in the $\mathcal{X}^{\alpha} -$norm, see
    Definition \ref{def:2dnoise}. This result is also proved in Proposition
    2.1 in {\cite{MZ}}.
  \end{proof}

We collect all relevant results about the map $\Gamma$.
\begin{lemma}\label{lem:Gamma}{\tmdummy}
  \begin{enumerateroman}
    \item $\Gamma : \mathcal{H}^s \rightarrow \mathcal{H}^s$ is bounded and
    invertible for any $s \in [0, 1)$
    
    \item $\Gamma : L^p \rightarrow L^p$ is bounded and invertible for any $p
    \in [2, \infty]$
    
    \item $\Gamma : \mathcal{H}^1 \rightarrow \mathcal{D} \left( \sqrt{- H}
    \right)$ is bounded and invertible.
    
    \item $\Gamma : \mathcal{H}^2 \rightarrow \mathcal{D} (H)$ is bounded and
    invertible.
    
    \item One has the bounds
    \begin{equation}
      \| (\Gamma - \tmop{id}) v \|_{\mathcal{H}^{1 - \varepsilon - s}}
      \lesssim \| v \|_{\mathcal{H}^{- s}} \quad \text{for } s \in (0, 1 -
      \varepsilon)  \text{and } \varepsilon > 0 \label{reg:G-1}
    \end{equation}
    and
    \begin{equation}
      \| \Gamma v - v - v \prec X \|_{\mathcal{H}^{1 + \sigma - \varepsilon}}
      \lesssim {\| v \|_{\mathcal{H}^{\sigma}}}  \text{ for } \sigma \in (0, 1
      - \varepsilon)  \text{and } \varepsilon > 0 \label{reg:BX}
    \end{equation}
    where both bounds remain true if we replace the Sobolev spaces
    $\mathcal{H}^s$ by Besov-H{\"o}lder spaces $\mathcal{C}^s .$
  \end{enumerateroman}
\end{lemma}
  \begin{proof}
    Everything but ii. and v. was proved in Section 2.1.1 of {\cite{GUZ}}. The
    cases $p = 2, \infty$ were also already proved. For a different $p$ we
    note that the result follows by interpolation. To prove v., one simply
    observes that in the case \eqref{reg:G-1} the dominant term is $u \prec X$
    which has precisely that bound and in the latter case \eqref{reg:BX} one
    just has to consider how $B_{\Xi} (v)$ is bounded for $v \in
    \mathcal{H}^{\sigma}$ and use the paraproduct estimates from the appendix.
  \end{proof}

Lastly we prove a statement about the ``sharpened'' group, which is the
transformation of the unitary group $e^{i t H}$
\[ e^{{-i t H^{\sharp}} } \assign \Gamma^{- 1} e^{-i t H} \Gamma . \]
It is clear that one still has the group property (even though the unitarity
is lost) since
\[ e^{{-i t H^{\sharp}} } e^{{-i s H^{\sharp}} } = \Gamma^{- 1} e^{-i t H} \Gamma
   \Gamma^{- 1} e^{-i s H} \Gamma = \Gamma^{- 1} e^{-i (t + s) H} \Gamma = e^{-i
   (t + s) {H^{\sharp}} } . \]
and one has the bounds for all times $t \in \mathbb{R}$
\begin{eqnarray*}
  \| e^{-i t H} u \|_{L^2} & \lesssim & \| u \|_{L^2}\\
  \| e^{-i t H} u \|_{\mathcal{D} (H)} & \lesssim & \| u \|_{\mathcal{D} (H)}
  .\\
  \left\| H^{\frac{s}{2}} e^{-i t H} u \right\|_{L^2} & \lesssim & \left\|
  H^{\frac{s}{2}} u \right\|_{L^2} \quad \text{for } s \in \mathbb{R}
\end{eqnarray*}
We have the analogous results for the transformed group.

\begin{lemma}
  \label{lem:Hsbound}For $s \in [0, 2]$ we get the following at any time $t
  \in \mathbb{R}$
  \[ \| e^{-i t H^{\sharp}} v \|_{\mathcal{H}^s} \lesssim \| v
     \|_{\mathcal{H}^s} \]
    \end{lemma}
  \begin{proof}
    See Proposition 2.2 in {\cite{MZ}}.
  \end{proof}

We finish this section by recalling the definition of the enhanced noise space
$\mathcal{X}^{\alpha}$ and the fact that smooth regularisations of the white
noise $\xi$ which can be defined as a random series on $\mathbb{T}^2$
\begin{equation}
  \xi (\omega) = \underset{n \in \mathbb{Z}^2}{\sum} g_n (\omega) e_n \
  \text{with } e_n  \text{ the Fourier basis and } g_n = \overline{g_{- n}} 
  \text{ i.i.d complex Gaussians} \label{def:white}
\end{equation}
converges in the $\mathcal{X}^{\alpha}$ topology, see e.g.
{\cite{allez_continuous_2015}} for details.

\begin{definition}
  \label{def:2dnoise}Let $\alpha = 1 + \kappa'$ for $0 < \kappa' \ll 1,$then
  we define the metric space
  \begin{eqnarray*}
    \mathcal{X}^{\alpha} & \assign & \overline{\{ (g, (1 - \Delta)^{- 1} g
    \circ g - a) : g \in C^{\infty} (\mathbb{T}^2), a \in \mathbb{R} \}}
    |_{\mathcal{C}^{- \alpha} \times \mathcal{C}^{2 - 2 \alpha}}
  \end{eqnarray*}
  which we call the enhanced noise space.
  
  For a smooth regularisation $\xi_{\varepsilon} = \xi \ast
  \rho_{\varepsilon}$ with smooth standard mollifier $\rho_{\varepsilon}$ and
  we set
  \begin{equation}
    \Xi_{\varepsilon}^2 \assign (1 - \Delta)^{- 1} \xi_{\varepsilon} \circ
    \xi_{\varepsilon} - c_{\varepsilon}
  \end{equation}
  for $c_{\varepsilon} =\mathbb{E} ((1 - \Delta)^{- 1} \xi_{\varepsilon} \circ
  \xi_{\varepsilon}) \sim \log \left( \frac{1}{\varepsilon} \right)$ a
  diverging constant.
\end{definition}

\begin{lemma}
  The lift of the regularised noise, $(\xi_{\varepsilon},
  \Xi_{\varepsilon}^2)$ converges to a limit $(\xi, \Xi^2)$ in
  $\mathcal{X}^{\alpha}$ in probability.
  
  In particular, from now on, we can use that the limit objects pathwise have
  the following regularities
  \begin{equation}
    \xi \in \mathcal{C}^{- 1 - \kappa'}, \xi \diamondsuit (1 - \Delta)^{- 1}
    \xi = \Xi^2 \in \mathcal{C}^{- 2 \kappa'} \quad \text{for } 0 < \kappa'
    \ll 1 \text{small.} \label{defXi2}
  \end{equation}
\end{lemma}

\begin{remark} \label{rem:genpot}
	Clearly a generic element in $\mathcal{C}^{-\alpha}$ will not have a lift in $\mathcal{X}^{\alpha}$, however, note that any potential in $\mathcal{C}^{-1+\kappa}$ does. Also, by using the Besov embedding, Lemma
  \ref{lem:besovem}, we can see that $V \in L^2 \hookrightarrow \mathcal{C}^{-
  1 - \kappa}$ and $(1 - \Delta)^{- 1} V \in \mathcal{H}^2
  \hookrightarrow C^{1 - \kappa}$ so $(V \circ (1 - \Delta)^{- 1} V) \in
  \mathcal{H}^{1 - \kappa} {\hookrightarrow \mathcal{C}^{- 2
  \kappa}} $ meaning that $L^2$ is canonically contained in the space
  $\mathcal{X}^{\alpha}$ so in principle all our results would be valid for
  $L^2$ potentials $V$, and in principle one can push this further but this
  would lead to a higher loss in regularity in the Strichartz estimate. In a
  recent paper {\cite{quasimode}}, Huang and Sogge proved Strichartz estimates
  like \eqref{eqn:classStr} for $- \Delta + V$ for $V \in L^{1 + \delta}$ so
  our result would not immediately imply theirs although in their paper it is
  not clear whether $V$ may be chosen as a distribution. It would be
  interesting to see if both approaches could be combined.
\end{remark}

\section{Strichartz estimates for the Anderson Hamiltonian}\label{sec:strand}

In this section we prove Theorem \ref{thm:2dandstr(intro)}.

\begin{proposition}
  \label{prop:2dmild}We have the following identity for a regular function,
  say $v \in \mathcal{H}^2$, and at any time $t \in \mathbb{R}$:
  \begin{equation}
    (e^{- itH^{\sharp}} - e^{- it \Delta}) v = -i \int^t_0 e^{- i (t - s)
    \Delta}  ((H^{\sharp} - \Delta) (e^{- isH^{\sharp}} v)) ds,
    \label{eqn:difference}
  \end{equation}
  moreover, fixing some $t_0 \in \mathbb{R}$, we get the related result
  \begin{equation}
    (e^{- i (t - t_0) H^{\sharp}} - e^{- i (t - t_0) \Delta}) v =- i
    \int^t_{t_0} e^{- i (t - s) \Delta}  ((H^{\sharp} - \Delta) (e^{- i (s -
    t_0) H^{\sharp}} v)) ds, \label{eqn:differencet0}
  \end{equation}
  where we recall, from Proposition \ref{prop:str:2dpert}, that there is a
  cancellation between $H^{\sharp}$ and the Laplacian.
  
  Moreover, on the interval $[t_0, t_1]$ with $| t_0 - t_1 | \leqslant 1$, we
  have for any small $\delta > 0$ the bound
  \begin{equation}
    \|(e^{- i (t - t_0) H^{\sharp}} - e^{- i (t - t_0) \Delta}) v\|_{L_{t ;
    [t_0, t_1]}^{\infty} \mathcal{H}^{\sigma}} \lesssim |t_1 - t_0 |
    \|v\|_{\mathcal{H}^{\sigma + 1 + \delta}} \label{eqn:prop11}
  \end{equation}
  for $\sigma \in [0, 1 - \delta)$.
  
  Also, for $r \geqslant 4$ we have
  \begin{eqnarray}
    \|(e^{- i (t - t_0) H^{\sharp}} - e^{- i (t - t_0) \Delta}) v\|_{L_{t ;
    [t_0, t_1]}^r W^{\sigma, r}} & \lesssim & \int^{t_1}_{t_0} \|e^{- i (s -
    t_0) H^{\sharp}} v\|_{\mathcal{H}^{\sigma + 2 - \frac{4}{r} + \delta}} 
    \label{eqn:prop3}\\
    & \lesssim & |t_1 - t_0 | \|v\|_{\mathcal{H}^{\sigma + 2 - \frac{4}{r} +
    \delta}}  \label{eqn:prop33}
  \end{eqnarray}
  for $\sigma \ge 0$ s.t. $\sigma + 2 - \frac{4}{r} + \delta \le 2$.
\end{proposition}

  \begin{proof}
    To prove \eqref{eqn:difference}, note that the l.h.s. solves a PDE. Set
    $v_1 (t) = e^{- it \Delta} v$, \ $v_2 (t) = e^{- itH^{\sharp}} v$ and
    $\bar{v} = v_1 - v_2$. Then
    \begin{eqnarray*}
      (i \partial_t - \Delta) v_1 & = & 0\\
      v_1 (0) & = & v\\
      (i \partial_t - \Delta) v_2 & = & (H^{\sharp} - \Delta) v_2\\
      v_2 (0) & = & v\\
      (i \partial_t - \Delta) \bar{v} & = & - (H^{\sharp} - \Delta) v_2\\
      \bar{v} (0) & = & 0
    \end{eqnarray*}
    From this we deduce that the mild formulation for $\bar{v}$ reads
    \[ \bar{v} (t) = -i \int^t_0 e^{- i (t - s) \Delta} (H^{\sharp} -
       \Delta) (v_2 (s)) ds \]
    which is \eqref{eqn:difference}. To prove \eqref{eqn:differencet0}, we
    proceed as above, with the difference that we replace $t$ by $t - t_0$ and
    do a change of variables in the integral.
    
    The bound \eqref{eqn:prop11} is clear using Lemma \ref{lem:Hsbound}. For
    the bound \eqref{eqn:prop3}, we apply first \eqref{eqn:difference} then we
    use the inhomogeneous Strichartz estimate from Proposition
    \ref{cor:scaling} to the right hand side and then Proposition
    \ref{prop:str:2dpert} to bound the term inside the integral. Subsequently,
    \eqref{eqn:prop33} follows by applying Lemma \ref{lem:Hsbound} and noting
    that the integrand does not depend on $s$ any more.
  \end{proof}

Now we are able to combine the above results to get the first new result.

\begin{theorem}
  \label{thm:2dandstr}[2-D Anderson Strichartz] Let $r \geqslant 4$, $\sigma
  \ge 0, \delta > 0$ s.t. $\sigma + (1 + \delta) \left( 1 - \frac{3}{r}
  \right) < 1$.Then we have on a finite time interval $[0, T]$, $T \leqslant
  1$ the following bound
  \begin{equation}
    \|e^{- itH^{\sharp}} v\|_{L_{t ; [0, T]}^r W^{\sigma, r}} \lesssim
    \|v\|_{\mathcal{H}^{\sigma + (1 + \delta) \left( 1 - \frac{3}{r} \right)}}
    \label{eqn:betternewstr}
  \end{equation}
  and
  
  \begin{align}
    & \left| \int^t_0 e^{- i (t - s) H^{\sharp}} f (s) ds \right|_{L_{t ; [0,
    T]}^r W^{\sigma, r}} \lesssim \int^T_0 \|f (s)\|_{\mathcal{H}^{\sigma + (1
    + \delta) \left( 1 - \frac{3}{r} \right)}} ds  \label{eqn:betternewstr2}
  \end{align}
\end{theorem}
  \begin{proof}
    We start by proving \eqref{eqn:betternewstr} with $\sigma = 0$ and $r =
    4$. The general case follows by interpolation. By Proposition
    \ref{prop:2dmild} and the Strichartz estimates in Theorem
    \ref{thm:imprstr} we can write, setting $v_N = P_{\leqslant N} v$, $I
    \assign [t_0, t_1]$ a subinterval of length $\sim \frac{1}{N}$ and $\delta
    > 0$
    \begin{align*}
      P_{\leqslant N} e^{- itH^{\sharp}} v_N & = P_{\leqslant N} e^{- i (t -
      t_0) H^{\sharp}} e^{- it_0 H^{\sharp}} v_N \nonumber\\
      & = e^{- i (t - t_0) \Delta} P_{\leqslant N} e^{- it_0 H^{\sharp}} v_N
      + P_{\leqslant N}  (e^{- i (t - t_0) H^{\sharp}} - e^{- i (t - t_0)
      \Delta}) e^{- it_0 H^{\sharp}} v_N \nonumber\\
      & = e^{- i (t - t_0) \Delta} P_{\leqslant N} e^{- it_0 H^{\sharp}} v_N
      - i \int^t_{t_0} P_{\leqslant N} e^{- i (t - s) \Delta} (H^{\sharp} -
      \Delta) (e^{- i (s - t_0) H^{\sharp}} v_N) ds. 
    \end{align*}
    Now we decompose the time interval into slices $I_j=[t^j_0,t^j_1]$ s.t. $\cup_j I_j = [0, T]$ with
    $|I_j | \sim \frac{1}{N}$
    
    \begin{align*}
      \|P_{\le N} e^{- itH^{\sharp}} v_N \|^4_{L_{t ; [0, T]}^4 L^4} & =
      \sum_{I_j = [t^j_0, t^j_1]} \|P_{\le N} e^{- itH^{\sharp}} v_N
      \|^4_{L_{t ; I_j}^4 L^4}\\
      & \lesssim \sum_{I_j = [t^j_0, t^j_1]} \|e^{- it \Delta} P_{\le N} e^{-
      it^j_0 H^{\sharp}} v_N \|^4_{L_{t ; I_j}^4 L^4} + \\&+\|P_{\le N} (e^{- i (t
      - t^j_0) H^{\sharp}} - e^{- i (t - t^j_0) \Delta}) e^{- it^j_0
      H^{\sharp}} v_N \|^4_{L_{t ; I_j}^4 L^4}\\
      & \lesssim \sum_{I_j = [t^j_0, t^j_1]} \|e^{- it^j_0 H^{\sharp}} v_N
      \|^4_{\mathcal{H}^{\delta}} + (\star)\\
      & \lesssim \sum_{I_j = [t^j_0, t^j_1]} \|v_N
      \|^4_{\mathcal{H}^{\delta}} + (\star)\\
      & \lesssim N^{1 + 4 \delta} \|v_N \|^4_{L^2} + \sum_{I_j = [t^j_0,
      t^j_1]} (\star) .
    \end{align*}
    
    Here we have used \eqref{eqn:differencet0} in each subinterval and applied
    the triangle inequality from the first to the second line. In the next
    step we have used the short-time bound from Proposition \ref{cor:scaling}
    and lastly Lemma \ref{lem:Hsbound} and the fact that there are $\sim N$
    summands allow us to conclude.
    
    \
    
    Next, we treat the perturbative part which we called $(\star)$
    
    \begin{align*}
      \|P_{\le N} (e^{- i (t - t_0) H^{\sharp}} - e^{- i (t - t_0) \Delta})
      e^{- it_0 H^{\sharp}} v_N \|^4_{L_{t ; I_j}^4 L^4} & = \left|
      \int^t_{t^j_0} e^{- i (t - s) \Delta} P_{\le N} (H^{\sharp} - \Delta)
      (e^{- i (s - t_0^j) H^{\sharp}} v_N) ds \right|^4_{L_{t ; I_j}^4 L^4}\\
      & \lesssim \left( \int_{I_j} \|P_{\le N} (H^{\sharp} - \Delta) (e^{- i
      (s - t_0^j) H^{\sharp}} v_N)\|_{\mathcal{H}^{\frac{\delta}{2}}} ds
      \right)^4\\
      & \lesssim \left( \int_{I_j} \|e^{- i (s - t_0^j) H^{\sharp}} v_N
      \|_{\mathcal{H}^{1 + \delta}} ds \right)^4\\
      & \lesssim \left( \int_{I_j} \|v_N \|_{\mathcal{H}^{1 + \delta}} ds
      \right)^4\\
      & \lesssim | I_j |^4 N^1 \|v_N \|_{\mathcal{H}^{\delta}}^4,\\
      & \lesssim N^{4 \delta} \|v_N \|_{L^2}^4
    \end{align*}
    
    having used the second bound in Proposition \ref{cor:scaling} to get to
    the second line and thereafter Proposition \ref{prop:str:2dpert}, Lemma
    \ref{lem:Hsbound} and Bernstein's inequality, Lemma \ref{lem:bernstein}.\\
    Thus we can conclude
    \[ \|P_{\le N} e^{- itH^{\sharp}} v_N \|^4_{L_{t ; [0, T]}^4 L^4} \lesssim
       N^{1 + 4 \delta} \|v_N \|^4_{L^2} \]
    hence
    \[ \|P_{\le N} e^{- itH^{\sharp}} P_{\leq N} v\|_{L_{t ; [0, T]}^4 L^4}
       \lesssim N^{\frac{1}{4} + \delta} \|P_{\leq N} v\|_{L^2} \]
    for any $\delta > 0$, which directly implies the result. The case $\sigma
    > 0$ is analogous since on this level it basically corresponds to
    multiplying with a power of $N$and for general $r \geqslant 4$ we
    interpolate between the bounds
    \begin{eqnarray*}
      \|P_{\le N} e^{- itH^{\sharp}} P_{\leq N} v\|_{L_{t ; [0, T]}^4 L^4} &
      \lesssim \|P_{\leq N} v\|_{\mathcal{H}^{\frac{1}{4} + \delta}}\\
      & \text{and}\\
      \|P_{\le N} e^{- itH^{\sharp}} P_{\leq N} v\|_{L_{t ; [0, T]}^{\infty}
      L^{\infty}} & \lesssim \|P_{\leq N} v\|_{\mathcal{H}^{1 + \delta}}
    \end{eqnarray*}
    where the latter is simply the trivial bound obtained from the Sobolev
    embedding $\mathcal{H}^{1 + \delta} \hookrightarrow L^{\infty} .$
    
    The inhomogeneous Strichartz estimate \eqref{eqn:betternewstr2} follows
    from the first in the usual way.
  \end{proof}

\section{Well-posedness of multiplicative stochastic NLS}\label{sec:solving}

\subsection{Low-regularity solutions}\label{sec:lwp}

We turn our attention to ``low-regularity'' solutions to the stochastic NLS
\begin{equation}
 \begin{aligned}
 (i \partial_t - H) u &=  - u | u |^{2 n}  \\
 u(0)&=u_0 \in \mathcal{H}^{\sigma}
 \end{aligned}\label{eqn:snls}
\end{equation}
for some $\sigma<1,$ which is formally
\[ (i \partial_t - \Delta) u = u \cdummy \xi + \infty u - u | u |^{2 n} . \]
In {\cite{GUZ}} this PDE was studied in the ``high regularity'' regime,
meaning $u_0 \in \mathcal{D} (H)$ or $\mathcal{D} \left( \sqrt{- H} \right) .$
Now we employ the Strichartz estimates to solve it in spaces of lower
regularity. In particular now we solve it in a space that does \textit{not}
depend on the realisation of the noise $\xi .$

Since this result is analogous to the result from {\cite{MZ}} on general
smooth surfaces where only the cubic case was considered, we only sketch the
proof and point out where the value of $n$ enters.

\begin{theorem}
  \label{thm:2dlwp}[LWP below energy space] For $\sigma \in \left( 1 -
  \frac{1}{2n}, 1 \right)$ \eqref{eqn:snls} is LWP in $\mathcal{H}^{\sigma} .$
  More precisely, there exists a short time $T > 0$ which is of size $T \sim
  (1 + \| u_0 \|_{\mathcal{H}^{\sigma}})^{- K}$ for some $K$ depending on $n$
  s.t. there exists a unique solution to
  \begin{equation}
    u (t) = e^{- i t H} u_0 - i \int^t_0 e^{- i (t - s) H} u | u |^{2 n} (s) d
    s
  \end{equation}
  in the space $C_{[0, T]} \mathcal{H}^{\sigma} \cap L_{[0, T]}^{2 n}
  W^{\frac{1}{n} + \kappa, 2 n}$ for $\kappa > 0$ sufficiently small where the
  solution in this norm depends continuously on the initial data in
  $\mathcal{H}^{\sigma} .$
\end{theorem}
  \begin{proof}
    By applying $\Gamma^{- 1}$to both sides and renaming both $\Gamma^{- 1}
    u_0 = u_0^{\sharp}$ and $\Gamma^{- 1} u = u^{\sharp}$ this becomes
    \[ u^{\sharp} (t) = e^{- i t H^{\sharp}} u^{\sharp}_0 - i \int^t_0 e^{- i
       (t - s) H^{\sharp}} \Gamma^{- 1} ((\Gamma u^{\sharp}) | \Gamma
       u^{\sharp} |^{2 n}) (s) d s. \]
    We want to show that this equation has a solution for a short time by
    setting up a fixed point argument in the space
    \[ C_{[0, T]} \mathcal{H}^{\sigma} \cap L_{[0, T]}^{2 n + \kappa} W^{\frac{1}{n}
       + \kappa, \frac{2 n}{1 - \kappa}}, \]
    where $T > 0$ is chosen later and $\kappa > 0$ is small enough that
    \[ \frac{1}{n} + \kappa + (1 + \kappa) \left( 1 - \frac{3 (1 - \kappa)}{2
       n} \right) \leqslant \sigma \]
    Then we bound, using several times Lemma \ref{lem:Gamma} and the
    Strichartz estimate Theorem \ref{thm:2dandstr}

    \begin{eqnarray*}
      \| u^{\sharp} \|_{L_{[0, T]}^{2 n + \kappa} W^{\frac{1}{n} + \kappa,
      \frac{2 n}{1 - \kappa}}} & \lesssim & \| u_0^{\sharp}
      \|_{\mathcal{H}^{\sigma}} + \int^T_0 \| \Gamma^{- 1} (\Gamma u^{\sharp}
      | \Gamma u^{\sharp} |^{2 n}) (\tau) \|_{\mathcal{H}^{\sigma}} d \tau\\
      & \lesssim & \| u_0^{\sharp} \|_{\mathcal{H}^{\sigma}} + \int^T_0 \|
      \Gamma u^{\sharp} | \Gamma u^{\sharp} |^{2 n} (\tau)
      \|_{\mathcal{H}^{\sigma}} d \tau\\
      & \lesssim & \| u_0^{\sharp} \|_{\mathcal{H}^{\sigma}} + \int^T_0 \|
      \Gamma u^{\sharp} (\tau) \|^{2 n}_{L^{\infty}} \| \Gamma u^{\sharp}
      (\tau) \|_{\mathcal{H}^{\sigma}} d \tau\\
      & \lesssim & \| u_0^{\sharp} \|_{\mathcal{H}^{\sigma}} + \| u^{\sharp}
      \|^{2 n}_{L_{[0, T]}^{2 n} L^{\infty}} \| u^{\sharp} \|_{L_{[0,
      T]}^{\infty} \mathcal{H}^{\sigma}}\\
      & \lesssim & \| u_0^{\sharp} \|_{\mathcal{H}^{\sigma}} + T^{\kappa} \|
      u^{\sharp} \|^{2 n}_{L_{[0, T]}^{2 n + \kappa} W^{\frac{1}{n} + \kappa,
      \frac{2 n}{1 - \kappa}}} \| u^{\sharp} \|_{L_{[0, T]}^{\infty}
      \mathcal{H}^{\sigma}} .
    \end{eqnarray*}
    For the other term we bound
    \begin{eqnarray*}
      \| u^{\sharp} \|_{L_{[0, T]}^{\infty} \mathcal{H}^{\sigma}} & \lesssim &
      \| u_0^{\sharp} \|_{\mathcal{H}^{\sigma}} + \int^T_0 \| \Gamma^{- 1}
      (\Gamma u^{\sharp} | \Gamma u^{\sharp} |^{2 n}) (\tau)
      \|_{\mathcal{H}^{\sigma}} d \tau\\
      & \lesssim & \| u_0^{\sharp} \|_{\mathcal{H}^{\sigma}} + T^{\kappa} \|
      u^{\sharp} \|^{2 n}_{L_{[0, T]}^{2 n + \kappa} W^{\frac{1}{n} + \kappa,
      \frac{2 n}{1 - \kappa}}} \| u^{\sharp} \|_{L_{[0, T]}^{\infty}
      \mathcal{H}^{\sigma}} .
    \end{eqnarray*}
    From here we can get a contraction for small times in the usual way.
    
    Thus we solve the sharpened equation and by applying $\Gamma$ we get a
    solution to the original equation. Observe that $u^{\sharp} {\in
    W^{\frac{1}{n} + \kappa, \frac{2 n}{1 - \kappa}}} $ implies that $(\Gamma
    - 1) u^{\sharp} {\in W^{1 - \kappa, \frac{2 n}{1 + \kappa}}}  $and thus
    ${u \in W^{\frac{1}{n} + \kappa, \frac{2 n}{1 - \kappa}}}  .$
  \end{proof}

\begin{remark}
  The fact that $s < 1$ as opposed to $s \geqslant 1$ makes the bound for the
  term $\Gamma^{- 1} (u | u |^2)$ easier since the paraproducts and other
  correction terms are actually more regular than $u$ and $u^{\sharp} .$ 
\end{remark}

\begin{remark}
  If we used the Strichartz estimates proved in {\cite{MZ}} which are
  \[ \| e^{- i t H} v \|_{L_{[0, 1]}^p L^q} \lesssim \| v
     \|_{\mathcal{H}^{\frac{1}{p} + \delta}} \quad \text{for } \frac{1}{p} +
     \frac{1}{q} = \frac{1}{2} \quad \text{and } \delta > 0 \quad
     \text{small,} \]
  we would get exactly the same condition \ $\sigma > 1 - \frac{1}{2n}$ for the
  wellposedness with the only difference that the function spaces in the
  contraction would be a bit different i.e. one would use $L_{[0, T]}^r L^s$
  spaces with different parameters $r, s$.
\end{remark}

\subsection{Global well-posedness in the energy space }

Now we turn to proving global in time well-posedness of
\begin{equation}
  \begin{aligned}
    (i \partial_t - H) u & =  - u |u|^{2 n}  \text{on } \mathbb{T}^2\\
    u (0) & =  u_0
  \end{aligned} \label{eqn:pdeenergy}
\end{equation}
\[ \  \]
in the \textit{energy space} $\mathcal{D}(\sqrt{-H})$. We recall that this is the space of functions 

\[ v \in L^2  : \text{\quad} \| \Gamma^{- 1} v \|_{\mathcal{H}^1} \sim \|
   v \|_{\mathcal{D} \left( \sqrt{- H} \right)} = | (- u, H u) |^{\frac{1}{2}}
   < \infty, \]
see Lemma \ref{lem:Gamma} and Lemma \ref{lem:adj}.\\ 
It is natural to consider solutions in this space, since the equation has the conserved energy
\[ E (u (t)) \assign -\frac{1}{2} (u, H u) + \frac{1}{2 n + 2} \int | u |^{2 n
	+ 2} = E (u_0) \]
	which controls the $\|u\|_{L^\infty_{[0,T]}\mathcal{D}(\sqrt{-H})}\sim\|u^\sharp\|_{L^\infty_{[0,T]}\mathcal{H}^1}$ norm. Using this bound, one can get global weak solutions as was done in \cite{GUZ} Section 3.2.2 in the cubic case but one does not obtain uniqueness and continuity on the data.\\
	The main difficulty in this setting, opposed to the low-regularity and strong regimes, is that it does not seem possible to make sense of the nonlinear term in the mild formulation i.e. bounding $\|\int_{0}^{t}e^{-i(t-s)H}|u|^{2n}u(s)ds\|_{\mathcal{D}(\sqrt{-H})}\sim\|\int_{0}^{t}e^{-i(t-s)H^\sharp}\Gamma^{-1}|\Gamma u^\sharp|^{2n}\Gamma u^\sharp(s)ds\|_{\mathcal{H}^1}$ in terms of the energy norm. One can get uniqueness from the Strichartz estimates as we will see, but \\
	Since we have a flow for the sharpened equation on $\mathcal{H}^\sigma$ for $\sigma\in(1-\frac{1}{2n},1)\cup\{2\},$ it would seem natural to try to \textquotedblleft interpolate\textquotedblright in order to extend the flow to $\mathcal{H}^1.$ Luckily, this is a straightforward consequence of a \textit{nonlinear interpolation} result from the recent work \cite{nonlinearint}. We cite the version of the main theorem which we will apply. 
\begin{theorem}\label{thm:nonlinearint}[Nonlinear interpolation]
	Let $R, T > 0$ and $s_0 < s < s_1$, moreover let
	\[ u_0 \in \mathcal{H}^s (\mathbb{T}^d) \quad B_s (u_0, R) \assign \{ w \in \mathcal{H}^s
	(\mathbb{T}^d) : \| w - u_0 \|_{\mathcal{H}^s} < R \} \]
	and assume that we have a (nonlinear) map
	\[ \Phi : B_s (u_0, R) \rightarrow L^{\infty} ([0, T] ; \mathcal{H}^{s_0}
	(\mathbb{T}^d)) \]
	satisfying the following properties
	\begin{itemize}
		\item \textbf{Weak Lipschitz bound}: There exists a constant $C_0>0$ (that
		may depend on $R, T, u_0$) s.t.
		\begin{equation}
			\text{for all } v_0, w_0 \in B_s (u_0, R) \quad \| \Phi (v_0) - \Phi
			(w_0) \|_{L_{[0, T]}^{\infty} \mathcal{H}^{s_0}} \leqslant C_0 \| v_0 - w_0
			\|_{\mathcal{H}^{s_0}} .
		\end{equation}
		\item \textbf{Tame estimate}: There exists a constant $C_1>0$ (that may
		depend on $R, T, u_0$) s.t.
		\begin{equation}
			\text{for all } v_0 \in C^{\infty} (\mathbb{T}^d) \cap B_s (u_0, R)
			\quad \| \Phi (v_0) \|_{L_{[0, T]}^{\infty} \mathcal{H}^{s_1}} \leqslant C_1 \|
			v_0 \|_{\mathcal{H}^{s_1}} .
		\end{equation}
		\item \textbf{Continuitiy in time}
		\[ \text{for all } v_0 \in B_s (u_0, R) \qquad \Phi (v_0) \in C ([0, T] ;
		\mathcal{H}^{s_0} (\mathbb{T}^d)). \]
	\end{itemize}
	Then we have $\text{for all } v_0 \in B_s (u_0, R)$ we have $\Phi (v_0) \in
	C ([0, T] ; \mathcal{H}^s (\mathbb{T}^d))$ and
	\[ \quad \Phi : v_0 \ni B_s (u_0, R) \mapsto \Phi (v_0) \in C ([0, T] ; \mathcal{H}^s
	(\mathbb{T}^d)) \]
	is continuous.
\end{theorem}

\begin{proof}
	This is a special case of Theorem 18 in \cite{nonlinearint}, in fact it is almost the same
	as Theorem 1, see Remark 4 in \cite{nonlinearint} that says exactly that Theorem 1 in \cite{nonlinearint} also holds in
	the periodic setting.
\end{proof}

We now apply this directly to the ``sharpened'' flow of the stochastic NLS
with $s_0 = 0, s = 1$ and $s_1 = 2$ and $u_0 = 0$.

In fact, we define $\Phi (v_0)$ as the unique fixed point of the map
\[ \Psi (v) (t) \assign e^{-i t H^{\sharp}} v_0 - i \int^t_0 e^{-i (t - s)
	H^{\sharp}} \Gamma^{- 1} (| \Gamma v |^{2 n} \Gamma v (s)) d s \]
in the space $C ([0, T^\star] ; \mathcal{H}^{1 - \kappa} (\mathbb{T}^2))$ where $v_0 \in \mathcal{H}^1$
and $T^\star \sim \| v_0 \|^{- K}_{\mathcal{H}^{1 - \kappa}}$for some $K > 0.$ This was shown to exist in Section \ref{sec:lwp} and by the energy bound (see the first point of Proposition \ref{prop:gronwall}) if $v_0\in\mathcal{H}^1$, then we can instead solve up to the smaller time $T' \sim \| v_0 \|^{- K}_{\mathcal{H}^{1}}$ and restart the flow up to time $2T'$ etc. So we have in particular that $\Phi$ satisfies the third point in Theorem \ref{thm:nonlinearint} for a generic time $T>0$\\
The other two properties can be proved by combining the Strichartz estimates and Gronwall's inequality in similar ways.
\begin{proposition}\label{prop:gronwall}
Let $\Phi,R$ be as above. Then we have
\begin{itemize}
	\item The following energy inequality holds:
	\begin{align}
	 \| \Phi (v_0) \|^2_{L_{[0, T]}^{\infty} \mathcal{H}^1 (\mathbb{T}^2)} \lesssim E
	(v_0) \lesssim R^2 . \end{align}
	\item One has the weak Lipschitz bound
	\begin{align}
		\| \Phi (v_0) - \Phi (w_0) \|_{L_{[0, T]}^{\infty} L^2} \leqslant C
		(E (v_0), T) \| v_0 - w_0 \|_{L^2}
	\end{align}
	\item One has the tame estimate 
	\begin{align}
	\| \Phi (v_0) (t) \|_{\mathcal{H}^2} \lesssim e^{TC'(E(v_0))}C(\|v_0\|_{\mathcal{H}^2})\label{eqn:expH2}
	\end{align}
\end{itemize}
\end{proposition}
\begin{proof}
The first point follows for nice enough $v_0$ by energy conservation and the norm equivalence $\|u^\sharp\|_{\mathcal{H}^1}\sim\|\Gamma u^\sharp\|_{\mathcal{D}(\sqrt{H})}.$ For $v_0\in\mathcal{H}^1$ it follows by approximation, see also \cite{GUZ}.\\
To prove the second point, one uses that one has the $L^2$ inequality 
\begin{eqnarray*}
	\frac{d}{d t} \| \Phi (v_0) - \Phi (w_0) \|^2_{L^2} (t) & \lesssim & (\| |
	\Phi (v_0) (t) |^{2 n} \|_{L^{\infty} (\mathbb{T}^2)} + \| | \Phi (w_0) (t)
	|^{2 n} \|_{L^{\infty} (\mathbb{T}^2)}) \| \Phi (v_0) - \Phi (w_0)
	\|^2_{L^2} (t)\\
	& \Rightarrow & \\
	\| \Phi (v_0) - \Phi (w_0) \|^2_{L^2} (t) & \lesssim & e^{C \int^t_0 (\| |
		\Phi (v_0) |^{2 n} (s) \|_{L^{\infty} (\mathbb{T}^2)} + \| | \Phi (w_0) |^{2
			n} (s) \|_{L^{\infty} (\mathbb{T}^2)}) {ds}} \| v_0 - w_0 \|^2_{L^2} 
\end{eqnarray*}
where the first bound follows by inserting the equation and the second is an application of Gronwall's inequality.\\
In order to control the term in the exponential, we recall the Strichartz estimate from Theorem \ref{thm:2dandstr}
\[ \| e^{-i t H^{\sharp}} g \|_{L_{[0, 1]}^4 W^{\sigma, 4} (\mathbb{T}^2)}
\lesssim \| g \|_{\mathcal{H}^{\frac{1}{4} + \sigma + \kappa}} \sigma \geqslant 0,
\kappa > 0 \text{ and } \frac{1}{4} + \sigma + \kappa < 1 \]
which, by the energy bound, implies
\[ \| \Phi (v_0) \|_{L_{[0, 1]}^4 W^{\sigma, 4}} \lesssim \| v_0 \|_{\mathcal{H}^1} + \|
v_0 \|^{2 n + 1}_{\mathcal{H}^1} \]
using that every term in the mild formulation is controlled by the $\mathcal{H}^1$ norm.\\
Thus we get the bound for small $\varepsilon,\delta>0$
\begin{eqnarray*}
	\int^1_0 \| | \Phi (v_0) |^{2 n} (s) \|_{L^{\infty} (\mathbb{T}^2)} d s &
	\lesssim & \int^1_0 \| | \Phi (v_0) |^{2 n} (s) \|_{W^{\frac{1}{2} + \delta,
			4 - \varepsilon} (\mathbb{T}^2)} d s\\
	& \lesssim & \int^1_0 \| \Phi (v_0) (s) \|_{L^p}^{2 n - 1} \| \Phi (v_0)
	(s) \|_{W^{\frac{1}{2} + \delta, 4} (\mathbb{T}^2)} d s\\
	& \lesssim & \| \Phi (v_0) (s) \|_{L_{[0, 1]}^{\infty} L^p}^{2 n - 1} \|
	\Phi (v_0) \|_{L_{[0, 1]}^4 W^{\frac{1}{2} + \delta, 4} (\mathbb{T}^2)}\\
	& \lesssim & C (E (v_0))
\end{eqnarray*}
for $p$ large enough depending on $\varepsilon$, having used the Sobolev
embedding and fractional Leibnitz rule, Lemma \ref{lem:fracleib}. By iterating this, we get
\[ \int^T_0 \| | \Phi (v_0) |^{2 n} (s) \|_{L^{\infty} (\mathbb{T}^2)} d s
\lesssim T C (E (v_0)) \]
This means, we have
\begin{eqnarray*}
	\| \Phi (v_0) - \Phi (w_0) \|_{L_{[0, 1]}^{\infty} L^2(\mathbb{T}^2)} & \leqslant & C (E
	(v_0)) \| v_0 - w_0 \|_{L^2(\mathbb{T}^2)} .
\end{eqnarray*}
and similarly
\[ \begin{array}{lll}
	\| \Phi (v_0) - \Phi (w_0) \|_{L_{[0, T]}^{\infty} L^2} & \leqslant & C
	(E (v_0), T) \| v_0 - w_0 \|_{L^2}
\end{array} \]
for general $T > 0.$ \\

Lastly, to prove the tame estimate, one proceeds similarly to \cite{GUZ} (where also all these steps were justified rigorously), note first the bound
\[ \| \Phi (v_0) (t) \|_{\mathcal{H}^2} \lesssim \| \partial_t \Phi (v_0) (t) \|_{L^2}
+ C (E (u_0)) \]
which follows from the equation
and, using the mild formulation,
\begin{eqnarray*}
	\partial_t \Phi (v_0) (t) & = & -e^{-i t H^{\sharp}} i H^{\sharp} v_0 +
	\int^t_0 e^{i (t - s) H^{\sharp}} \partial_s \Gamma^{- 1} (| \Gamma \Phi
	(v_0) |^{2 n} \Gamma \Phi (v_0) (s)) d s \\&&+ \Gamma^{- 1} (| \Gamma \Phi (v_0)
	|^{2 n} \Gamma \Phi (v_0) (t))\\
	& \Rightarrow & \\
	\| \partial_t \Phi (v_0) (t) \|_{L^2} & \leqslant & C (\| v_0 \|_{H^2}) +
	\int^t_0 \| \partial_t \Phi (v_0) (s) \|_{L^2} \| | \Phi (v_0) |^{2 n} (s)
	\|_{L^{\infty}} d s
\end{eqnarray*}
having used the energy bound on the last term and the $L^2$ boundedness of the sharpened group. Ultimately, we have by Gronwall again
\begin{align*}
	\| \partial_t \Phi (v_0) (t) \|_{L^2}  \leqslant & C (\| v_0 \|_{H^2})
	e^{\int^t_0 \| | \Phi (v_0) |^{2 n} (s) \|_{L^{\infty}} d s}\\
	 \leqslant & C (\| v_0 \|_{H^2}) e^{T C' (E (v_0))}
\end{align*}
which is precisely what we wanted to show.
\end{proof}\\
This concludes the proof of global well-posedness in
the energy space, since we have shown that $\Phi$ satisfies the assumptions of Theorem \ref{thm:nonlinearint}.\\
\begin{remark}
Even in the case $n = 1$, which was treated in \cite{debussche2018schrodinger} and \cite{GUZ}, the bound \eqref{eqn:expH2} is an
improvement since there one has an iterated exponential bound in time whereas
this gives a simple exponential bound.\\
In Section \ref{sec:2ndorder} we will actually  that for $v_0 \in \mathcal{H}^2$ and $\| v_0
\|_{\mathcal{H}^1} < R$ we have
\[ \| \Phi (v_0) \|_{L_{[0, T]}^{\infty} H^2} \lesssim C (R, T) \| v_0
\|_{H^2} \]
where the constant $C$ is growing only polynomially in $T$.
\end{remark}
\subsection{Global well-posedness of strong solutions and growth of norms}\label{sec:2ndorder}

We consider the equation
\begin{equation}
  \begin{aligned}
    (i \partial_t - H) u & = - u | u |^{2 n}\\
    u (0) & = u_0
  \end{aligned}. \label{nls:strong}
\end{equation}
for general $n$ and $u_0 \in \mathcal{D} (H)$ which we call the
\textit{strong regime.} This is equivalent to solving
\begin{equation}
  (i \partial_t - H^{\sharp}) u^{\sharp} = - \Gamma^{- 1} (\Gamma u^{\sharp} |
  \Gamma u^{\sharp}|^{2n}) \label{nls:strongsharp}
\end{equation}
with initial data $u^{\sharp} (0) = \Gamma^{- 1} u_0 \in \mathcal{H}^2 .$

The first order conserved energy is given by
\[ E (u (t)) \assign -\frac{1}{2} (u, H u) + \frac{1}{2 n + 2} \int | u |^{2 n
   + 2} = E (u_0) \]
which controls the norm $\| u \|_{L_{[0, T]}^{\infty} \mathcal{D} \left(
\sqrt{H} \right)} \sim \| u^{\sharp} \|_{L_{[0, T]}^{\infty}
\mathcal{H}^1}$.

We recall from {\cite{GUZ}} that in this regime we have local-in-time
wellposedness from a contraction argument in the mild formulation
\eqref{eqn:intromild} following on the one hand from the formal norm equivalence
$\| u \|_{L_{[0, T]}^{\infty} \mathcal{D} (H)} \sim \| \partial_t u
\|_{L_{[0, T]}^{\infty} L^2}$, see Lemma \ref{lem:strongH}, and the
observation that (see {\cite{GUZ}} for details)
\[ \left\| \int^t_0 e^{-i (t - s) H} u | u |^{2 n - 2} (s) d s \right\|_{L_{t ;
   [0, T]}^{\infty} \mathcal{D} (H)} \lesssim \| \partial_t (u | u |^{2 n -
   2}) \|_{L_{[0, T]}^{\infty} \mathcal{D} (H)} \lesssim \| \partial_t u
   \|_{L_{[0, T]}^{\infty} \mathcal{D} (H)} \| u \|^{2 n - 2}_{L_{[0,
   T]}^{\infty} L^{\infty}} \]
and the embedding $\mathcal{D} (H) \hookrightarrow L^{\infty}$, since one can
not apply $H$ to a nonlinear term but one may integrate by parts in the time
integral to instead.

\begin{remark}
  In {\cite{GUZ}}, Remark 3.9, it was wrongly claimed that global-in-time
  well-posedness follows in a similar way for general powers. Local
  well-posedness, however, can be proved in the same way for general powers.
\end{remark}

The main result of this section is the following, which says that we can
extend these solutions to all times but we also get a bound on the growth of
Sobolev bounds as in {\cite{PTV}}

\begin{theorem}
  \label{thm:normgrowth}Let $T > 0$ and $\kappa > 0$ small. We can find an
  almost conserved energy $E^{(1)} (u)$ for which one has
  \[ \left| E^{(1)} (u (t)) - \frac{1}{2} \int | \partial_t u (t) |^2 \right|
     \lesssim C (u_0) (\| \partial_t u (t) \|^{\kappa}_{L^2} + 1) \quad
     \forall t \in [0, T] \]
  and which has growth
  \[ \underset{0 \leqslant t \leqslant T}{\sup} | E^{(1)} (u (t)) | \lesssim
     T^{8 + \kappa} . \]
\end{theorem}

This is based on the approach from {\cite{PTV}} and is a comparable result to
what was shown in {\cite{TVa}} and {\cite{TVb}} but uses the paracontrolled
approach from {\cite{GUZ}} and the Strichartz estimates as in {\cite{MZ}}.
Moreover, the growth of the norm is not shown in those papers.

\

For the second order energy that should control the norm $\| u \|_{L_{[0,
T]}^{\infty} \mathcal{D} (H)} \sim \| u^{\sharp} \|_{L_{[0, T]}^{\infty}
\mathcal{H}^2}$ or equivalenty $\| \partial_t u \|_{L_{[0, T]}^{\infty} L^2}
\sim \| \partial_t u^{\sharp} \|_{L_{[0, T]}^{\infty} L^2}$. We collect these
norm equivalences in a small lemma.

\begin{lemma}
  \label{lem:strongH}Let $u$ be a solution of \eqref{nls:strong} and
  $\Gamma^{- 1} u = u^{\sharp}$ solve \eqref{nls:strongsharp}. Then we have that
  the following bounds hold
  \begin{align}
    \| u \|_{L_{[0, T]}^{\infty} \mathcal{D} (H)} & \lesssim \| u^{\sharp}
    \|_{L_{[0, T]}^{\infty} \mathcal{H}^2} \lesssim \| u \|_{L_{[0,
    T]}^{\infty} \mathcal{D} (H)} \\
    \| \partial_t u \|_{L_{[0, T]}^{\infty} L^2} & \lesssim \| \partial_t
    u^{\sharp} \|_{L_{[0, T]}^{\infty} L^2} \lesssim \| \partial_t u
    \|_{L_{[0, T]}^{\infty} L^2} \\
    & \lesssim \| u \|_{L_{[0, T]}^{\infty} \mathcal{D} (H)} + C (E (u_0)) 
    \label{bound:dtH}\\
    \| H u \|_{L_{[0, T]}^{\infty} \mathcal{H}^{- \alpha}} & \lesssim \|
    u^{\sharp} \|_{L_{[0, T]}^{\infty} \mathcal{H}^{2 - \alpha}} \lesssim \|
    \partial_t u \|_{L_{[0, T]}^{\infty} \mathcal{H}^{- \alpha}} + C (E (u_0))
    \text{ for } \alpha \in (0, 1)  \label{bound:neg1}\\
    \| \partial_t u \|_{L_{[0, T]}^{\infty} \mathcal{H}^{- \alpha}} & \lesssim
    \| \partial_t u^{\sharp} \|_{L_{[0, T]}^{\infty} \mathcal{H}^{- \alpha}}
    \lesssim \| H u \|_{L_{[0, T]}^{\infty} \mathcal{H}^{- \alpha}} + C (E
    (u_0))  \text{ for } \alpha \in (0, 1)   \label{bound:neg2}\\
    \| u^{\sharp} \|_{L_{[0, T]}^{\infty} \mathcal{H}^{\gamma}} & \lesssim C
    (E (u_0)) \| u^{\sharp} \|^{\gamma - 1}_{L^{\infty} \mathcal{H}^2} \quad
    \text{for } \gamma \in [1, 2].  \label{interpolateH}
  \end{align}
  Moreover, the difference
  \begin{equation}
    \partial_t u - \partial_t u^{\sharp} = (\Gamma - 1) \partial_t u^{\sharp}
  \end{equation}
  satisfies the bound
  \begin{equation}
    \| (\Gamma - 1) \partial_t u^{\sharp} \|_{\mathcal{H}^{\alpha}} \lesssim
    \| \partial_t u^{\sharp} \|_{\mathcal{H}^{\alpha - 1 + \kappa}} \quad
    \text{for } \alpha < 1 \quad \text{and } \kappa < 1 - \alpha
    \label{diffdt}
  \end{equation}
  similarly we have
  \begin{equation}
    \| (\Gamma - 1) u^{\sharp} \|_{L_{[0, T]}^{\infty} \mathcal{C}^{1 - 2
    \kappa}} \lesssim \| u^{\sharp} \|_{L_{[0, T]}^{\infty} \mathcal{C}^{-
    \kappa}} \lesssim \| u^{\sharp} \|_{L_{[0, T]}^{\infty} \mathcal{H}^1}
    \lesssim C (E (u_0)) \quad \text{for } \kappa > 0. \label{diffusharp}
  \end{equation}
\end{lemma}

\begin{proof}
  The first two equivalences follow from the properties of the $\Gamma$ map
  from Lemma \ref{lem:Gamma}. To get \eqref{bound:dtH}, we use the equation to
  get the bound
  \[ \| \partial_t u \|_{L_{[0, T]}^{\infty} L^2} \leqslant \| H u \|_{L_{[0,
     T]}^{\infty} L^2} + \| u \|^{2 n + 1}_{L_{[0, T]}^{\infty} L^{2 (2 n +
     1)}} \lesssim \| H u \|_{L_{[0, T]}^{\infty} L^2} + E^{\frac{2 n + 1}{2}}
     (u_0) \]
  where the last step follows from the embedding $\mathcal{D} \left( \sqrt{-
  H} \right) \hookrightarrow L^{2 (2 n + 1)}$ , see Lemma \ref{lem:Gamma}, and
  the energy conservation in Proposition \ref{prop:gronwall}. The two bounds
  \eqref{bound:neg1}, \eqref{bound:neg2} follow as above combined with Theorem
  \ref{thm:Gammaemb} and \eqref{interpolateH} follows from
  $\mathcal{H}^{\gamma}$ interpolation and the energy bounding $\| u^{\sharp}
  \|_{L_{[0, T]}^{\infty} \mathcal{H}^1} .$
  
  Finally, \eqref{diffdt} and \eqref{diffusharp} follow from Lemma
  \ref{lem:Gamma}. 
\end{proof}

We next use the Strichartz estimate, Theorem \ref{thm:2dandstr}
\[ \| e^{i t H^{\sharp}} v \|_{L_{[0, 1]}^4 W^{\alpha, 4}} \lesssim \| v
   \|_{\mathcal{H}^{\alpha + \frac{1}{4} + \kappa}}  \text{for } \kappa > 0
   \text{ and } \alpha + \frac{1}{4} + \kappa \leqslant 2  \]
and obtain some bounds as a consequence.

\begin{lemma}
  \label{lem:strong:str}Let $u$ be a solution of \eqref{nls:strong} and
  $\Gamma^{- 1} u = u^{\sharp}$ solve \eqref{nls:strongsharp}. Then we have
  \begin{align}
    \| u \|_{L_{[0, 1]}^4 C^{\frac{1}{4} - \kappa}} \lesssim & C (E (u_0))
    \text{ for any } \kappa > 0  \label{str:u} \\
    \| | u |^m u \|_{L_{[0, 1]}^4 C^{\alpha}} \lesssim & C (E (u_0)) \| u
    \|_{L_{[0, 1]}^4 C^{\alpha + \kappa}} \quad \text{for } m \in \mathbb{N}
    \text{ and } \alpha + \kappa < 1  \label{polyu}\\
    & \text{in particular} \nonumber\\
    \| | u |^m u \|_{L_{[0, 1]}^4 C^{\frac{1}{4} - \kappa}} \lesssim & C (E
    (u_0)) \text{\quad for any } \kappa > 0 \\
    & \text{and} \nonumber\\
    \| \partial_t | u |^m \|_{L_{[0, 1]}^4 \mathcal{H}^{- \alpha}} \lesssim &
    C (E (u_0)) \| \partial_t u \|_{L_{[0, 1]}^{\infty} \mathcal{H}^{-
    \alpha}}  \text{ for } \alpha \in \left[ 0, \frac{1}{4} \right) 
    \label{polyut}\\
    \| u^{\sharp} \|_{L_{[0, 1]}^4 W^{\frac{7}{4} - \kappa, 4}} \lesssim & C(E(u_0)) (1 + \| \partial_t u \|_{L_{[0, 1]}^{\infty}
    L^2})  \text{ for } \kappa > 0 . \label{bound:pde74}
  \end{align}
\end{lemma}

\begin{proof}
  Using the Sobolev embedding $W^{\beta + \frac{1}{2} + \varepsilon, 4}
  \hookrightarrow C^{\beta}$ for $\varepsilon > 0$ and the Strichartz
  estimate, we get
  \begin{align*}
    \| u \|_{L_{[0, 1]}^4 C^{\frac{1}{4} - \kappa}} & \lesssim \| u \|_{L_{[0,
    1]}^4 W^{\frac{3}{4} - \frac{\kappa}{2}, 4}}\\
    & \lesssim \| u_0 \|_{\mathcal{H}^{1 - \frac{\kappa}{4}}} + \int^1_0 \| u
    | u |^{2 n} (s) \|_{\mathcal{H}^{1 - \frac{\kappa}{4}}} d s\\
    & \lesssim \| u_0 \|_{\mathcal{H}^{1 - \frac{\kappa}{4}}} + \int^1_0 \| u
    (s) \|_{\mathcal{H}^{1 - \tilde{\kappa}}}^{2 n + 1} d s\\
    & \lesssim \| u_0 \|_{\mathcal{D} \left( \sqrt{- H} \right)} + \| u
    \|_{L_{[0, 1]}^{\infty} \mathcal{D} \left( \sqrt{- H} \right)}^{2 n + 1}\\
    & \lesssim C (E (u_0))
  \end{align*}
  where $1 \gg \kappa \gg \tilde{\kappa} > 0$ are small enough and can be
  determined using the fractional Leibnitz inequality, Lemma
  \ref{lem:fracleib}. This proves \eqref{str:u}.
  
  For \eqref{polyu}, we use the Sobolev embedding, for every $\kappa >
  0$ there exists $p \gg 1$ s.t. $W^{\alpha + \kappa, p} \hookrightarrow
  C^{\alpha}$ and fractional Leibnitz, Lemma \ref{lem:fracleib},
  \[ \| | u |^m u \|_{L_{[0, 1]}^4 C^{\alpha}} \lesssim \| | u |^m u
     \|_{L_{[0, 1]}^4 W^{\alpha + \kappa, p}} \lesssim \| u \|^m_{L_{[0, 1]}^{\infty} L^{m p}} \| u
     \|_{L_{[0, 1]}^4 C^{\alpha + \kappa}}\lesssim C (E
     (u_0))\| u
     \|_{L_{[0, 1]}^4 C^{\alpha + \kappa}} \]
  giving \eqref{polyu}, so in particular $\| | u |^m u \|_{L_{[0, 1]}^4
  C^{\frac{1}{4} - \kappa}} \lesssim C (E (u_0))$ as claimed. This also
  implies immediately the bound \eqref{polyut} by the Leibnitz rule for functions/distributions.
  
  Finally, we want to prove \eqref{bound:pde74} which would be obvious in the
  classical case, but needs an additional argument since this is the first
  time we want to take more than one derivative in the Strichartz estimate. We
  observe that for general functions $f \in C_{[0, T]}^1 L^2$ we have by
  integrating by parts in time
  \begin{align*}
   \int^t_0 e^{-i (t - s) H^{\sharp}} f (s) d s &= \frac{1}{ i} e^{-i t
     H^{\sharp}} \int^t_0 (H^{\sharp})^{- 1} \left( \frac{d}{d s} e^{ i s
     H^{\sharp}} \right) f (s) d s \\
      &= \frac{1}{i} \int^t_0 e^{-i (t - s)
     H^{\sharp}} (H^{\sharp})^{- 1} \frac{d}{d s} f (s) d s + \frac{1}{ i}
     (H^{\sharp})^{- 1} (f (t) - e^{-i t H^{\sharp}} f (0)) . 
    \end{align*}
  Then, using the mild formulation for $u^{\sharp}$ (set $\Gamma^{- 1}
  ((\Gamma u^{\sharp}) | \Gamma u^{\sharp} |^{2 n}) = f$ for readability)
  \begin{align*}
    u^{\sharp} (t) & = e^{- i t H^{\sharp}} u^{\sharp}_0 + i \int^t_0 e^{- i
    (t - s) H^{\sharp}} f (s) d s\\
    & = e^{- i t H^{\sharp}} u^{\sharp}_0 + \int^t_0 e^{- i (t - s)
    H^{\sharp}} (H^{\sharp})^{- 1} \frac{d}{d s} f (s) d s - (H^{\sharp})^{-
    1} (f (t) - e^{i t H^{\sharp}} f (0))
  \end{align*}
  so we get the bound from the Strichartz estimate Theorem \ref{thm:2dandstr}
  and Lemma \ref{lem:Gamma}
  \begin{align*}
    \| u^{\sharp} \|_{L_{[0, 1]}^4 W^{\frac{7}{4} - \kappa, 4}}  \lesssim &
    \| u_0^{\sharp} \|_{\mathcal{H}^2} + \| (H^{\sharp})^{- 1} \Gamma
    \partial_t f \|_{L_{[0, 1]}^1 \mathcal{H}^2} + \| (H^{\sharp})^{- 1} f
    \|_{L_{[0, 1]}^4 W^{\frac{7}{4} - \kappa, 4}} + \\&+\| e^{i t H^{\sharp}}
    (H^{\sharp})^{- 1} \Gamma^{- 1} f (0) \|_{L_{[0, 1]}^4 W^{\frac{7}{4} -
    \kappa, 4}}\\
     \lesssim & \| u_0^{\sharp} \|_{\mathcal{H}^2} + C (E (u_0)) \|
    \partial_t u \|_{L_{[0, 1]}^{\infty} L^2} + \| u | u |^{2 n} \|_{L_{[0,
    1]}^{\infty} \mathcal{H}^{\frac{1}{4} - \kappa}} + \| u_0 | u_0 |^{2 n}
    \|_{L^2}\\
     \lesssim & C (E (u_0)) \| \partial_t u \|_{L_{[0, 1]}^{\infty} L^2} + C
    ( E(u_0) )\\
     \lesssim & C (E(u_0)) (1 + \| \partial_t u \|_{L_{[0, 1]}^{\infty} L^2})
  \end{align*}
  finishing the proof.
\end{proof}

\

The general approach, following {\cite{PTV}}, is to make a simple initial
ansatz $E^0 = \frac{1}{2} \int | \partial_t u |^2 $for the almost conserved
energy $E^{(1)}$ and then compute its time derivative and try to write
\[ \frac{d}{d t} E^0 = \frac{d}{d t} A + O ((E^0)^{1 -}), \]
so terms on the rhs should either be total time derivatives or be ``lower
order'' i.e. can be bounded by sublinear terms of $E^0 .$ Usually this will
result from interpolating between the conserved energy $E (u)$ which is
equivalent to the $\mathcal{H}^1$ norm of $u^{\sharp}$ and the almost
conserved energy $E^0$ which is comparable to the $\mathcal{H}^2$ norm of
$u^{\sharp}$ or equivalently the $L^2$ norm of $\partial_t u$, i.e. we will apply \eqref{interpolateH}.

\begin{proof}[of Theorem \ref{thm:normgrowth}]
  We make an ansatz for the almost conserved energy
  \[ E^0 (\partial_t u) = \frac{1}{2} \mathrm{Re} \int | \partial_t u |^2 =
     \frac{1}{2} (\partial_t u, \partial_t u) \]
  then we make the straightforward computation which is taking a time
  derivative and inserting the equation
  \begin{eqnarray*}
    \frac{d}{d t} E^0 (\partial_t u) & = & \mathrm{Re} \int \overline{\partial_t
    u} \partial^2_t u\\
    & = & - \mathrm{Re }\ i \int \overline{\partial_t u} \partial_t (H u - | u
    |^{2 n} u)\\
    & = & - \mathrm{Re }\ i \int \overline{\partial_t u} \partial_t (| u |^{2 n}
    u)\\
    & = & - \mathrm{Re}\ i \int \overline{\partial_t u}  (\partial_t u | u |^{2
    n} + u \partial_t | u |^{2 n - 2})\\
    & = & - \mathrm{Re}\ i \int \overline{\partial_t u} u \partial_t | u |^{2 n
    - 2},
  \end{eqnarray*}
  having used also the self-adjointness of $H$ in the second line and the
  realness of the first expression in the penultimate line . Further we insert
  the equation again and obtain
  \begin{eqnarray}
    \frac{d}{d t} E^0 (\partial_t u) & = & \mathrm{Re} \int \overline{(H u + | u
    |^{2 n} u)} u \partial_t | u |^{2 n - 2} \nonumber\\
    & = & \mathrm{Re} \int \overline{H u} u \partial_t | u |^{2 n - 2} +
    \mathrm{Re} \int | u |^{2 n + 2} \partial_t | u |^{2 n - 2} \nonumber\\
    & = & \mathrm{Re} \int \overline{H \Gamma u^{\sharp}} \Gamma u^{\sharp}
    \partial_t | \Gamma u^{\sharp} |^{2 n - 2} + (n - 1) \int | u |^{4 n - 2}
    \partial_t | u |^2 \nonumber\\
    & = & \mathrm{Re} \int \overline{H \Gamma u^{\sharp}} \Gamma u^{\sharp}
    \partial_t | \Gamma u^{\sharp} |^{2 n - 2} + \frac{n - 1}{2 n} \frac{d}{d
    t} \int | u |^{4 n} \nonumber\\
    & \backassign & (I) + \frac{n - 1}{2 n} \frac{d}{d t} \int | u |^{4 n} 
    \label{eqn:I}
  \end{eqnarray}
  having inserted $\Gamma u^{\sharp} = u$ and the definition of the real part
  for the first term .
  
  \
  
  Now, the main point is that, up to more regular terms, the term
  $\overline{u} H u + u \overline{H u}$ in $(I)$ is comparable to $\Delta |
  u^{\sharp} |^2$ and the $\partial_t | \Gamma u^{\sharp} |^{2 n - 2}$ should
  be replaced with $\partial_t | u^{\sharp} |^{2 n - 2}$ and then one can
  proceed similarly to {\cite{PTV}}.\\
  In other words, we aim to prove the bound
  \begin{align}
  |\mathrm{Re} \int \overline{H \Gamma u^{\sharp}} \Gamma u^{\sharp}
  \partial_t | \Gamma u^{\sharp} |^{2 n - 2} -\mathrm{Re} \int \overline{\Delta u^{\sharp}} u^{\sharp}
  \partial_t |  u^{\sharp} |^{2 n - 2} 
  | \lesssim  \left( \| \partial_t u \|^{\gamma}_{L_{[0, T]}^{\infty} L^2}
  + 1 \right)\label{ineq: strategy}
  \end{align}
  for some $0<\gamma<2$ since the secon
  \
  
  To begin with, we want to replace the time derivative $\partial_t | \Gamma
  u^{\sharp} |^{2 n - 2}$ by $\partial_t | u^{\sharp} |^{2 n - 2}$ in $(I)$.
  We bound the difference as follows
  \begin{align}
    (\overline{u} H u, (\partial_t | \Gamma u^{\sharp} |^{2 n - 2} -
    \partial_t | u^{\sharp} |^{2 n - 2}))  = & (\overline{u} H u, (\Gamma -
    1) \partial_t u^{\sharp} \overline{u} | u |^{2 n - 4}) +
    (\overline{u} H u, \partial_t u (| u |^{2 n - 2} - | u^{\sharp} |^{2 n -
    2})) \nonumber\\
     = & (H u, | u |^{2 n - 2} (\Gamma - 1) \partial_t u^{\sharp}) +
    (H u, \partial_t u (\Gamma -1 )u^{\sharp} (P^{(2 n - 3)} (u, u^{\sharp}))) +\nonumber\\&+
    (H u,(\Gamma - 1)\overline{u}^{\sharp} \tilde{P}^{(2 n - 3)} (u,
    u^{\sharp})),  \label{dtusharp:diff}
  \end{align}
  where we have written
  \[ | u |^{2 n - 2} - | u^{\sharp} |^{2 n - 2} = (\Gamma-1) u^{\sharp} P^{2 n -
     3} (u, u^{\sharp}) + (\Gamma-1) \overline{u}^{\sharp} \tilde{P}^{(2
     n - 3)} (u, u^{\sharp}) \]
  for some polynomials $P^{2 n - 3}, \tilde{P}^{2 n - 3}$ of degree $2 n - 3$
  in $u, \overline{u}, u^{\sharp}$ and $\overline{u}^{\sharp}$. Now we start
  by bounding the first term in \eqref{dtusharp:diff} integrated out in time on an
  interval $[0, T]$ for some $T \leqslant 1$
  \begin{eqnarray*}
    | | (H u, | u |^{2 n - 2} (\Gamma - 1) \partial_t u^{\sharp}) |
    |_{L_{[0, T]}^4} & \lesssim & \| | u |^{2 n - 2} H u \|_{L_{[0, T]}^4
    \mathcal{H}^{- \frac{1}{4} + \delta}} {{\| (\Gamma - 1) \partial_t
    u^{\sharp} \|_{L_{[0, T]}^{\infty} \mathcal{H}^{\frac{1}{4} + \delta}}} }
    \\
    & \lesssim & \| | u |^{2 n - 2} \|_{L_{[0, T]}^4 C^{\frac{1}{4} -
    \frac{\delta}{2}}} \| H u \|_{L_{[0, T]}^{\infty} \mathcal{H}^{-
    \frac{1}{4} + \delta}} {{\| (\Gamma - \tmop{id}) \partial_t u^{\sharp}
    \|_{L_{[0, T]}^{\infty} \mathcal{H}^{\frac{1}{4} + \delta}}} } \\
    & \lesssim & C (E (u_0)) \left( \| u^{\sharp} \|_{L_{[0, T]}^{\infty}
    \mathcal{H}^{\frac{7}{4} + \delta}} \left\| \partial_t u^{\sharp}
    \right\|_{{{L_{[0, T]}^{\infty} \mathcal{H}^{- \frac{3}{4} + 2 \delta}}} }
    + 1 \right)\\
    & \lesssim & C (E (u_0)) \left( \| u^{\sharp} \|_{L_{[0, T]}^{\infty}
    \mathcal{H}^{\frac{7}{4} + \delta}} \| u^{\sharp} \|_{L_{[0,
    T]}^{\infty} \mathcal{H}^{\frac{5}{4} + 2 \delta}}  + 1 \right)\\
    & \lesssim & C (E (u_0)) \left( \| u^{\sharp} \|_{L_{[0,
    T]}^{\infty} \mathcal{H}^2}^{1 + 3 \delta}  + 1 \right)
  \end{eqnarray*}
  having used Lemmas \ref{lem:strongH} and \ref{lem:strong:str}. Next we bound
  the second term in \eqref{dtusharp:diff}, the third being completely analogous,
  this term is worse in terms of regularity and we split off the worst terms
  using paraproducts, see the appendix for the definition and recall the
  convention $\preccurlyeq \assign \prec + \circ$ we have
  \begin{align*}
    (H u, \partial_t u (\Gamma -1 )u^{\sharp} (P^{(2 n - 3)} (u, u^{\sharp}))) = &
    (H u, (\partial_t u (\Gamma -1 )u^{\sharp}) \preccurlyeq (P^{(2 n - 3)} (u,
    u^{\sharp}))) +\\&+ (H u, (\partial_t u (\Gamma -1 )u^{\sharp}) \succ (P^{(2 n - 3)}
    (u, u^{\sharp})))\\
     = & (H u, (\partial_t u (\Gamma -1 )u^{\sharp}) \preccurlyeq (P^{(2 n - 3)}
    (u, u^{\sharp}))) \\&+ (H u, (\partial_t u \preccurlyeq (\Gamma -1 )u^{\sharp})
    \succ (P^{(2 n - 3)} (u, u^{\sharp}))) \\&+ (H u, (\partial_t u \succ (\Gamma -1 )u^{\sharp}) \succ (P^{(2 n - 3)} (u, u^{\sharp})))
  \end{align*}
  which we bound as
  \begin{align*}
    | | \ldots | |_{L_{[0, T]}^4}  \lesssim & \| H u \|_{L_{[0, T]}^{\infty}
    \mathcal{H}^{- \frac{\delta}{4}}} \| \partial_t u \|_{L_{[0, T]}^{\infty}
    \mathcal{H}^{- \frac{1}{4} + \frac{\delta}{2}}} \| (\Gamma -1 )u^{\sharp}
    \|_{L_{[0, T]}^{\infty} C^{\frac{1}{4} - \frac{\delta}{4}}} \| P^{(2 n -
    3)} (u, u^{\sharp}) \|_{L_{[0, T]}^4 C^{\frac{1}{4} - \frac{\delta}{4}}} +\\&+
    \| (H u, (\partial_t u \preccurlyeq (\Gamma -1 )u^{\sharp}) \succ (P^{(2 n - 3)}
    (u, u^{\sharp}))) \|_{L_{[0, T]}^4} + \\&+ \| (H u, (\partial_t u \succ (\Gamma -1 )u^{\sharp}) \succ (P^{(2 n - 3)} (u, u^{\sharp}))) \|_{L_{[0, T]}^4}\\
     \lesssim & C (E(u_0)) \left( \| \partial_t u \|^{\frac{7}{4} +
    \frac{\delta}{4}}_{L_{[0, T]}^{\infty} L^2} + 1 \right) +\\&+ \| (\Gamma -1 )u^{\sharp} \|_{L_{[0, T]}^{\infty} C^{1 - \delta}} \| \partial_t u
    \|_{L_{[0, T]}^{\infty} L^2} \| P^{(2 n - 3)} (u, u^{\sharp}) \|_{L_{[0,
    T]}^4 L^{\infty}} \| H u \|_{L_{[0, T]}^{\infty} \mathcal{H}^{- 1 +
    \delta}} +\\&+ \| (H u, (\partial_t u \succ (\Gamma -1 )u^{\sharp}) \succ (P^{(2 n -
    3)} (u, u^{\sharp}))) \|_{L_{[0, T]}^4}\\
     \lesssim & C (E(u_0)) \left( \| \partial_t u \|^{\frac{7}{4} +
    \frac{\delta}{4}}_{L_{[0, T]}^{\infty} L^2} + 1 \right) + \| (\star)
    \|_{L_{[0, T]}^4} .
  \end{align*}
  Naively, the term $(\star)$ gives us a bound no better than $C (u_0) \|
  \partial_t u \|_{L_{[0, T]}^{\infty} L^2} \| H u \|_{L_{[0, T]}^{\infty}
  L^2} \sim \| \partial_t u \|^2_{L_{[0, T]}^{\infty} L^2}$ which would be
  enough if we wanted to apply Gronwall, but we are able to deal with it
  separately to get an improved bound, namely with only powers of $\|
  \partial_t u \|_{L^{\infty} L^2} $ strictly smaller than $2$ appearing.
  
  \
  
  By applying Corollary \ref{cor:adj2} to this term, we may bound it as
  \begin{align*}
  \| (\star) \|_{L_{[0, T]}^4} &\lesssim \| H u \|_{L_{[0, T]}^{\infty}
     \mathcal{H}^{- \frac{1}{8} + \delta}} \| \partial_t u \|_{L_{[0,
     T]}^{\infty} \mathcal{H}^{- \frac{1}{8} + \delta}} \| u - u^{\sharp}
     \|_{L_{[0, T]}^{\infty} C^{\frac{1}{4} - \frac{\delta}{2}}} \| P^{(2 n -
     3)} (u, u^{\sharp}) \|_{L_{[0, T]}^4 C^{\frac{1}{4} - \frac{\delta}{2}}}\\
     &\lesssim C (E(u_0)) \left( \| \partial_t u \|^{\frac{7}{4} + 2
     \delta}_{L_{[0, T]}^{\infty} L^2} + 1 \right) 
    \end{align*}
  This means in \eqref{eqn:I} the term $(I)$ we can replaced $\mathrm{Re} \int
  \overline{H \Gamma u^{\sharp}} \Gamma u^{\sharp} \partial_t | \Gamma
  u^{\sharp} |^{2 n - 2}$ by $\mathrm{Re} \int \overline{H \Gamma u^{\sharp}}
  \Gamma u^{\sharp} \partial_t | u^{\sharp} |^{2 n - 2}$ up to an error $\sim
  \left( \| \partial_t u \|^{\frac{7}{4} + 2 \delta}_{L_{[0, T]}^{\infty} L^2}
  + 1 \right)$.
  
  \
  
  Next, we want to replace $\mathrm{Re} \int \overline{H \Gamma u^{\sharp}}
  \Gamma u^{\sharp} \partial_t | u^{\sharp} |^{2 n - 2}$ by $\mathrm{Re} \int
  \overline{\Delta u^{\sharp}} u^{\sharp} \partial_t | u^{\sharp} |^{2 n - 2}$
  in $(I)$. In order to do this, we rewrite it as
  \begin{align*}
    &\mathrm{Re} \int \overline{H \Gamma u^{\sharp}} \Gamma u^{\sharp} \partial_t
    | u^{\sharp} |^{2 n - 2} - \mathrm{Re} \int \overline{\Delta u^{\sharp}}
    u^{\sharp} \partial_t | u^{\sharp} |^{2 n - 2}=\\
    &= \mathrm{Re} \int \overline{(H \Gamma - \Delta) u^{\sharp}} \Gamma
    u^{\sharp} \partial_t | u^{\sharp} |^{2 n - 2} - \mathrm{Re} \int
    \overline{\Delta u^{\sharp}}  (\Gamma - 1) u^{\sharp} \partial_t |
    u^{\sharp} |^{2 n - 2} 
  \end{align*}
  now we bound the first of these terms as, using Lemmas \ref{lem:Gamma},
  \ref{lem:strongH} and \ref{lem:strong:str}
  \begin{align*}
    \left\| \mathrm{Re} \int \overline{(H \Gamma - \Delta) u^{\sharp}} \Gamma
    u^{\sharp} \partial_t | u^{\sharp} |^{2 n - 2} \right\|_{L_{[0, T]}^4} &
    \lesssim  \| u \|_{L_{[0, T]}^4 C^{\frac{1}{4} - \frac{\delta}{2}}} \| (H
    \Gamma - \Delta) u^{\sharp} \|_{L_{[0, T]}^{\infty}
    \mathcal{H}^{\frac{1}{4} - \delta}} \| \partial_t | u^{\sharp} |^{2 n - 2}
    \|_{L_{[0, T]}^{\infty} \mathcal{H}^{- \frac{1}{4} + \delta}}\\
    & \lesssim  \| u \|_{L_{[0, T]}^4 C^{\frac{1}{4} - \frac{\delta}{2}}} \|
    u^{\sharp} \|_{L_{[0, T]}^{\infty} \mathcal{H}^{\frac{5}{4} -
    \frac{\delta}{2}}} \| \partial_t | u^{\sharp} |^{2 n - 2} \|_{L_{[0,
    T]}^{\infty} \mathcal{H}^{- \frac{1}{4} + \delta}}\\
    & \lesssim  C (E (u_0)) \| u^{\sharp} \|^{\frac{3}{4} -
    \frac{\delta}{2}}_{L_{[0, T]}^{\infty} H^2} \| \partial_t u^{\sharp}
    \|_{L^{\infty} \mathcal{H}^{- \frac{1}{4} + \delta}}\\
    & \lesssim  C (E (u_0)) \left( \| \partial_t u \|^{1 +
    \frac{\delta}{2}}_{L^{\infty} L^2} + 1 \right)
  \end{align*}
  and similarly the second one
  \begin{align*}
    \| (\Delta u^{\sharp}, (\Gamma - 1) u^{\sharp} \partial_t |
    u^{\sharp} |^{2 n - 2}) \|_{L_{[0, T]}^4} & \lesssim  \| (\Delta
    u^{\sharp}, ((\Gamma - 1) u^{\sharp}) \succcurlyeq \partial_t |
    u^{\sharp} |^{2 n - 2}) \|_{L_{[0, T]}^4} +\\&+ \| (\Delta u^{\sharp}, (\Gamma
    - 1) u^{\sharp} \prec \partial_t | u^{\sharp} |^{2 n - 2})
    \|_{L_{[0, T]}^4}\\
    & \lesssim  \| \Delta u^{\sharp} \|_{L_{[0, T]}^{\infty} \mathcal{H}^{-
    1 + \delta}} \| (\Gamma - 1) u^{\sharp} \|_{L_{[0, T]}^{\infty}
    C^{1 - \delta}} \| \partial_t | u^{\sharp} |^{2 n - 2} \|_{L_{[0, T]}^4
    L^2} +\\&+ \| (\Delta u^{\sharp}, (\Gamma - 1) u^{\sharp} \prec
    \partial_t | u^{\sharp} |^{2 n - 2}) \|_{L_{[0, T]}^4}\\
    & \lesssim  C (E (u_0)) \| u^{\sharp} \|^{\delta}_{L_{[0, T]}^{\infty}
    \mathcal{H}^2} \| \partial_t u \|_{L_{[0, T]}^{\infty} L^2} +\\&+ \| \Delta
    u^{\sharp} \|_{L_{[0, T]}^{\infty} \mathcal{H}^{- \frac{1}{8} + \delta}}
    \| (\Gamma - 1) u^{\sharp} \|_{L_{[0, T]}^{\infty} C^{1 - \delta}}
    \| \partial_t | u^{\sharp} |^{2 n - 2} \|_{L_{[0, T]}^4 \mathcal{H}^{-
    \frac{1}{8} + \delta}}\\
    & \lesssim  C (E (u_0)) \left( 1 + \| \partial_t u \|^{1 +
    \delta}_{L_{[0, T]}^{\infty} L^2} + \| u^{\sharp} \|_{L_{[0, T]}^{\infty}
    \mathcal{H}^{\frac{15}{8} + \delta}} \| \partial_t u^{\sharp} \|_{L_{[0,
    T]}^{\infty} \mathcal{H}^{- \frac{1}{8} + \delta}} \right)\\
    & \lesssim  C (E (u_0)) \left( \| \partial_t u \|^{\frac{7}{4} + 2
    \delta}_{L^{\infty} L^2} + 1 \right) .
  \end{align*}
  So we have successfully replaced the term $(I)$ by the term $\mathrm{Re} \int
  \overline{\Delta u^{\sharp}} u^{\sharp} \partial_t | u^{\sharp} |^{2 n -
  2}$, which we deal with as in {\cite{PTV}}. Indeed, by the Leibnitz rule
  \begin{eqnarray*}
    \mathrm{Re} \int \overline{\Delta u^{\sharp}} u^{\sharp} \partial_t |
    u^{\sharp} |^{2 n - 2} & = & 2 \int (\overline{\Delta u^{\sharp}}
    u^{\sharp} + \overline{u^{\sharp}} \Delta u^{\sharp}) \partial_t |
    u^{\sharp} |^{2 n - 2}\\
    & = & 2 (n - 1) \int  (\Delta | u^{\sharp} |^2 - 2 | \nabla u^{\sharp}
    |^2) \partial_t | u^{\sharp} |^2 | u^{\sharp} |^{2 n - 4}
  \end{eqnarray*}
  and the first term we rewrite as
  \begin{eqnarray*}
    \int \Delta | u^{\sharp} |^2 \partial_t | u^{\sharp} |^2 | u^{\sharp} |^{2
    n - 4} & = & \int - \nabla | u^{\sharp} |^2 \partial_t \nabla | u^{\sharp}
    |^2 | u^{\sharp} |^{2 n - 4} - \int \nabla | u^{\sharp} |^2 \partial_t |
    u^{\sharp} |^2 \nabla | u^{\sharp} |^{2 n - 4}\\
    & = & - \frac{d}{d t} \left( \int  \frac{1}{2} | \nabla | u^{\sharp} |^2
    |^2 | u^{\sharp} |^{2 n - 4} \right) + \frac{(2 - n)}{2} \int  | \nabla |
    u^{\sharp} |^2 |^2 \partial_t | u^{\sharp} |^2 | u^{\sharp} |^{2 n - 4}
  \end{eqnarray*}
  where the first term will be included in the energy and the second one we
  bound integrated in time as

  \begin{eqnarray*}
    \left\| \int  | \nabla | u^{\sharp} |^2 |^2 \partial_t | u^{\sharp} |^2 |
    u^{\sharp} |^{2 n - 4} \right\|_{L_{[0, T]}^2} & \lesssim & \|  | \nabla |
    u^{\sharp} |^2 |^2 \|_{L_{[0, T]}^2 L^{\frac{2}{1 - \delta}}} \|
    \partial_t u^{\sharp} \|_{L_{[0, T]}^{\infty} L^2} \| | u^{\sharp} |^{2 n
    - 1} \|_{L_{[0, T]}^{\infty} L^{\frac{2}{\delta}}}\\
    & \lesssim & C (E (u_0)) \| \partial_t u \|_{L_{[0, T]}^{\infty} L^2} \|
    \nabla | u^{\sharp} |^2 \|^2_{L_{[0, T]}^4 L^{\frac{4}{1 - \delta}}}\\
    & \lesssim & C (E (u_0)) \| \partial_t u \|_{L_{[0, T]}^{\infty} L^2} \|
    \nabla u^{\sharp} \|^2_{L_{[0, T]}^4 L^{\frac{4}{1 - \delta / 2}}}\\
    & \lesssim & C (E (u_0)) \| \partial_t u \|_{L_{[0, T]}^{\infty} L^2} \|
    u^{\sharp} \|^2_{L_{[0, T]}^4 W^{1 + \kappa, 4}}\\
    & \lesssim & C (E (u_0)) \| \partial_t u \|_{L_{[0, T]}^{\infty} L^2} \|
    u^{\sharp} \|^{\frac{1}{2} - \tilde{\kappa}}_{L_{[0, T]}^4 W^{\frac{3}{4}
    - \kappa, 4}} \| u^{\sharp} \|^{\frac{1}{2} + \tilde{\kappa}}_{L_{[0,
    T]}^4 W^{\frac{7}{4} - \kappa, 4}}\\
    & \lesssim & C (E (u_0)) \left( \| \partial_t u \|^{\frac{3}{2} +
    \tilde{\kappa}}_{L^{\infty} L^2} + 1 \right)
  \end{eqnarray*}
  for small $1 \gg \delta, \kappa (\delta), \tilde{\kappa} (\kappa) > 0$.
  
  In total, we define (as usual, $u=\Gamma u^\sharp$)
  \begin{equation}
    E^{(1)} (u) (t) \assign \frac{1}{2} \int | \partial_t u (t) |^2 - \frac{(n
    - 1)}{2} \int  | \nabla | u^{\sharp} (t) |^2 |^2 | u^{\sharp} (t) |^{2 n -
    2} - \frac{n - 1}{2 n} \int | u (t) |^{4 n}
  \end{equation}
  for which we have for any $t \in [0, T], \kappa > 0$
  \begin{align*}
    \left| E^{(1)} (u) (t) - \frac{1}{2} \int | \partial_t u (t) |^2 \right| &
    \lesssim C (E (u_0)) (\| \nabla | u^{\sharp} |^2 (t) \|_{L^{2 + \kappa}} +
    1)\\
    & \lesssim C (E (u_0)) (\| u^{\sharp} (t) \|^2_{\mathcal{H}^{1 + 2
    \kappa}} + 1)\\
    & \lesssim C (E (u_0)) (\| u^{\sharp} (t) \|^{4 \kappa}_{\mathcal{H}^2} +
    1)
  \end{align*}
  meaning that, up to lower order terms, we can treat $E^{(1)} (u) (t)$ as
  controlling $\| \partial_t u \|^2_{L_{[0, t]}^{\infty} L^2}$.
  
  Collecting all the previous computations, we have for $s < t$ with $| s - t
  | \leqslant 1$ since the Strichartz estimates only hold in short intervals, we have
  \begin{eqnarray*}
    \left| \int | \partial_t u (t) |^2 - \int | \partial_t u (s) |^2 \right| &
    \lesssim & | E^{(1)} (u) (t) - E^{(1)} (u) (s) | + C (E (u_0)) \|
    \partial_t u \|^{4 \kappa}_{L_{[s, t]}^{\infty} L^2}\\
    & = & \left| \int^t_s \frac{d}{d t} E^{(1)} (u (\tau)) d \tau \right| + C
    (E (u_0)) \| \partial_t u \|^{4 \kappa}_{L_{[s, t]}^{\infty} L^2}\\
    & \lesssim & C (E (u_0)) \left( | t - s |^{\frac{1}{2}} \left( 1 +
    \| \partial_t u \|_{{L_{[s, t]}^{\infty} L^2}}^{\frac{7}{4} +
    	\kappa}  \right) + \| \partial_t u \|^{4 \kappa}_{L_{[s, t]}^{\infty}
    L^2} \right) .
  \end{eqnarray*}
  We also get
  \begin{align*}
    \underset{\tau \in [s, t]}{\sup} \int | \partial_t u (\tau) |^2 &
    \leqslant \int | \partial_t u (s) |^2 + C (E (u_0)) \left( | t - s
    |^{\frac{1}{2}} \left( 1 +\| \partial_t u \|_{{L_{[s, t]}^{\infty} L^2}}^{\frac{7}{4} +
    	\kappa}   \right) + \| \partial_t u \|^{4
    \kappa}_{L_{[s, t]}^{\infty} L^2} \right)\\
    & \leqslant | \partial_t u (s) |^2 + \left( 1 + \| \partial_t u \|_{{L_{[s, t]}^{\infty} L^2}}^{\frac{7}{4} +
    	\kappa}   \right) + C
    (E (u_0)) \| \partial_t u \|^{4 \kappa}_{L_{[s, t]}^{\infty} L^2}\\
    &\text{for } 
     | t - s | = \left( \frac{1}{1 + C (E (u_0))} \right)^2
    \backassign \tau_0 < 1
  \end{align*}
  To conclude, for a general $T > 0$ we fix $N = \lceil \tau_0 T \rceil$ and
  thus we have
  \begin{eqnarray*}
    \underset{\tau \in [0, T]}{\sup} \int | \partial_t u (\tau) |^2 & = &
    \underset{i = 1, \ldots N}{\max} \underset{\tau \in \left[ \frac{i}{N} T,
    \frac{i + 1}{N} T \right]}{\sup} \int | \partial_t u (\tau) |^2\\
    & \leqslant & \int | \partial_t u_0 |^2 + N \left( 1 + \| \partial_t    u\|_{L_{[0, T]}^{\infty} L^2}^{\frac{7}{4} + \kappa}  + C (E
    (u_0)) \| \partial_t u \|^{4 \kappa}_{L_{[0, T]}^{\infty} L^2} \right)\\
    & = & \\
    & \leqslant & \int | \partial_t u_0 |^2 + C (E (u_0)) N^{1 +
    \tilde{\kappa}} + C N^{8 + \tilde{\kappa}} + \frac{1}{2} {\| \partial_t u
    \|_{L_{[0, T]}^{\infty} L^2}^2} 
  \end{eqnarray*}
  for $\tilde{\kappa} (\kappa) > 0$ as small as we want. This finally implies
  \begin{eqnarray*}
    \underset{\tau \in [0, T]}{\sup} \int | \partial_t u (\tau) |^2 & \lesssim
    & E^{(1)} (u_0) + C (E (u_0)) + T^{8 + \tilde{\kappa}}
  \end{eqnarray*}
  so we have polynomial growth as claimed.
\end{proof}

\appendix\section{Paracontrolled Distributions and Besov spaces}

We collect some elementary results about paraproducts, see
{\cite{gubinelli2015paracontrolled}}, {\cite{allez_continuous_2015}},
{\cite{bahouri2011fourier}} for more details. We work on the $d -$dimensional
torus \
\[ \mathbb{T}^d =\mathbb{R}^d /\mathbb{Z}^d \]
in the current paper always $d = 2$.

The Sobolev space $\mathcal{H}^{\alpha} (\mathbb{T}^d)$ with index $\alpha \in
\mathbb{R}$ is defined as
\[ \mathcal{H}^{\alpha} (\mathbb{T}^d) \assign \left\{ u \in \mathcal{S}'
   (\mathbb{T}^d) : \left\| (1 - \Delta)^{\frac{\alpha}{2}} u \right\|_{L^2} <
   \infty \right\} . \]
Next, we recall the definition of Littlewood-Paley blocks. We denote by $\chi$
and $\rho$ two non-negative smooth and compactly supported radial functions
$\mathbb{R}^d \rightarrow \mathbb{C}$ such that
\begin{enumerateroman}
  \item The support of $\chi$ is contained in a ball and the support of $\rho$
  is contained in an annulus $\{ x \in \mathbb{R}^d : a \leqslant | x |
  \leqslant b \}$
  
  \item For all $\xi \in \mathbb{R}^d$, $\chi (\xi) + \underset{j \geqslant
  0}{\sum} \rho (2^{- j} \xi) = 1$;
  
  \item For $j \geqslant 1$, $\chi (\cdummy) \rho (2^{- j} \cdummy) = 0$ and
  $\rho (2^{- j} \cdummy) \rho (2^{- i} \cdummy) = 0$ for $| i - j | > 1$.
\end{enumerateroman}
The Littlewood-Paley blocks $(\Delta_j)_{j \geqslant - 1}$ associated to $f
\in \mathcal{S}' (\mathbb{T}^d)$ are defined by
\[ \Delta_{- 1} f \assign \mathcal{F}^{- 1} \chi \mathcal{F}f \text{and }
   \Delta_j f \assign \mathcal{F}^{- 1} \rho (2^{- j} \cdummy) \mathcal{F}f
   \text{for} j \geqslant 0. \]
We also set, for $f \in \mathcal{S}' (\mathbb{T}^d)$ and $j \geqslant - 1$
\[ S_j f \assign \underset{i = - 1}{\overset{j - 1}{\sum}} \Delta_i f. \]
Then the \ Besov space with parameters $p, q \in [1, \infty), \alpha \in
\mathbb{R}$ can now be defined as \
\[ B_{p, q}^{\alpha} (\mathbb{T}^d) \assign \{ u \in \mathcal{S}'
   (\mathbb{T}^d) : \| u \|_{B_{p, q}^{\alpha}} < \infty \}, \]
where the norm is defined as
\[ \| u \|_{B_{p, q}^{\alpha}} \assign \left( \underset{k \geqslant - 1}{\sum}
   ((2^{\alpha k} \| \Delta_k u \|_{L^p})^q) \right)^{\frac{1}{q}}, \]
with the obvious modification for $q = \infty .$ We also define the
\textit{Besov-H{\"o}lder} spaces
\[ \mathcal{C}^{\alpha} \assign B_{\infty \infty}^{\alpha} \]
which for $\alpha \in (0, 1)$ agree with the usual H{\"o}lder spaces
$C^{\alpha} .$

Using this notation, we can formally decompose the product $f \cdummy g$ of
two distributions $f$ and $g$ as \
\[ f \cdummy g = f \prec g + f \circ g + f \succ g, \]
where
\[ f \prec g \assign \underset{j \geqslant - 1}{\sum} S_{j - 1} f \Delta_j g
   \quad \text{and\quad} f \succ g \assign \underset{j \geqslant - 1}{\sum}
   \Delta_j f S_{j - 1} g \]
are referred to as the \textit{paraproducts}, whereas
\[ f \circ g \assign \underset{j \geqslant - 1}{\sum} \underset{| i - j |
   \leqslant 1}{\sum} \Delta_i f \Delta_j g \]
is called the \textit{resonant product}. An important point is that the
paraproduct terms are always well defined whatever the regularity of $f$ and
$g$. The resonant product, on the other hand, is a priori only well defined if
the sum of their regularities is positive. We collect some results.

\begin{lemma}
  \label{lem:para}[cf.Theorem 3.17 {\cite{mourrat2017global}}]Let $\alpha,
  \alpha_1, \alpha_2 \in \mathbb{R}$ and $p, p_1, p_2, q \in [1, \infty]$ be
  such that
  \[ \alpha_1 \neq 0 \quad \alpha = (\alpha_1 \wedge 0) + \alpha_2 \quad
     \text{and} \quad \frac{1}{p} = \frac{1}{p_1} + \frac{1}{p_2} . \]
  Then we have the bound
  \[ \| f \prec g \|_{B_{p, q}^{\alpha}} \lesssim \| f \|_{B_{p_1,
     \infty}^{\alpha_1}} \| g \|_{{B^{\alpha_2}_{p_2, q}} } \]
  and in the case where $\alpha_1 + \alpha_2 > 0$ we have the bound
  \[ \| f \circ g \|_{B_{p, q}^{\alpha_1 + \alpha_2}} \lesssim \| f
     \|_{B_{p_1, \infty}^{\alpha_1}} \| g \|_{{B^{\alpha_2}_{p_2, q}} } . \]
\end{lemma}

\begin{remark}
  For the majority of the paper we care only about the case where $p = p_2 = q
  = 2$ and $p_1 = \infty$.
\end{remark}

\begin{lemma}
  \label{lem:bernstein}[Bernstein's inequality] Let $\mathcal{A}$ be an
  annulus and $\mathcal{B}$ be a ball. For any $k \in \mathbb{N}, \lambda >
  0,$and $1 \leqslant p \leqslant q \leqslant \infty$ we have
  \begin{enumerate}
    \item if $u \in L^p (\mathbb{T}^d) $is such that $\mathrm{supp}
    (\mathcal{F}u) \subset \lambda \mathcal{B}$ then
    \[ \underset{\mu \in \mathbb{N}^d : | \mu | = k}{\max} \| \partial^{\mu} u
       \|_{L^q} \lesssim_k \lambda^{k + d \left( \frac{1}{p} - \frac{1}{q}
       \right)} \| u \|_{L^p} \]
    \item if $u \in L^p (\mathbb{T}^d) $is such that $\mathrm{supp}
    (\mathcal{F}u) \subset \lambda \mathcal{A}$ then
    \[ \lambda^k \| u \|_{L^p} \lesssim_k \underset{\mu \in \mathbb{N}^d : |
       \mu | = k}{\max} \| \partial^{\mu} u \|_{L^p} . \]
  \end{enumerate}
\end{lemma}

\

\begin{lemma}
  \label{lem:besovem}[Besov embedding] Let $\alpha < \beta \in \mathbb{R}, $
  $q_1 \leqslant q_2$, and $p > r \in [1, \infty]$ be such that
  \[ \beta = \alpha + d \left( \frac{1}{r} - \frac{1}{p} \right), \]
  then we have the following bound
  \[ \| f \|_{B_{p, q_2}^{\alpha} (\mathbb{T}^d)} \lesssim \| f \|_{{B_{r,
     q}^{\beta}}_1 (\mathbb{T}^d)} . \]
\end{lemma}

\begin{proposition}
  \label{prop:commu}[Commutator Lemma, Proposition 4.3 in
  {\cite{allez_continuous_2015}}]
  
  Given $\alpha \in (0, 1)$, $\beta, \gamma \in \mathbb{R}$ such that $\beta +
  \gamma < 0$ and $\alpha + \beta + \gamma > 0$, the following trilinear
  operator $C$ defined for any smooth functions $f, g, h$ by
  \[ C (f, g, h) \assign (f \prec g) \circ h - f (g \circ h) \]
  can be extended continuously to the product space $\mathcal{H}^{\alpha}
  \times \mathcal{C}^{\beta} \times \mathcal{C}^{\gamma}$. Moreover, we have
  the following bound
  \[ || C (f, g, h) ||_{\mathcal{H}^{\alpha + \beta + \gamma - \delta}}
     \lesssim || f ||_{\mathcal{H}^{\alpha}} || g ||_{\mathcal{C}^{\beta}} ||
     h ||_{\mathcal{C}^{\gamma}} \]
  for all $f \in \mathcal{H}^{\alpha}$, $g \in \mathcal{C}^{\beta}$ and $h \in
  \mathcal{C}^{\gamma}$, and every $\delta > 0$.
\end{proposition}

\begin{lemma}
  \label{lem:fracleib}[Fractional Leibniz,{\cite{gulisashvili1996exact}}] Let
  $1 < p < \infty$ and $p_1, p_2, p'_1, p'_2$ such that
  \[ \frac{1}{p_1} + \frac{1}{p_2} = \frac{1}{p'_1} + \frac{1}{p'_2} =
     \frac{1}{p} . \]
  Then for any $s, \alpha \geqslant 0$ there exists a constant s.t.
  \[ \| \langle \nabla \rangle^s (f g) \|_{L^p} \leqslant C \| \langle \nabla
     \rangle^{s + \alpha} f \|_{L^{p_2}} \| \langle \nabla \rangle^{- \alpha}
     g \|_{L^{p_1}} + C \| \langle \nabla \rangle^{- \alpha} f \|_{L^{p'_2}}
     \| \langle \nabla \rangle^{s + \alpha} g \|_{L^{p'_1}} . \]
\end{lemma}

\begin{lemma}
  \label{lem:tame}[Tame estimate,Corollary 2.86 in
  {\cite{bahouri2011fourier}}] For any $s > 0$ and $(p, r) \in [1, \infty]^2,$
  the space $B^s_{p, r} \cap L^{\infty}$ is an algebra and the bound
  \[ \| u \cdummy v \|_{B_{p, q}^s} \lesssim \| u \|_{B_{p, q}^s} \| v
     \|_{L^{\infty}} + \| u \|_{L^{\infty}} \| v \|_{B_{p, q}^s} \]
  holds.
\end{lemma}

\section{Almost adjointness lemmas for paraproducts}

\begin{lemma}\label{lem:adj} [cf appendix of {\cite{GUZ}}]
  We have
  \[ D (f, g, h) \assign (f, g \circ h) - (f \prec g, h) \]
  defined on smooth functions, extends to a bounded map
  \[ B_{p_1 \infty}^{\alpha} \times B_{p_2 2}^{\beta} \times B_{p_3
     2}^{\gamma} \rightarrow \mathbb{R} \]
  for $1 = \frac{1}{p_1} + \frac{1}{p_2} + \frac{1}{p_3}$ and $\alpha + \beta
  + \gamma = 0$ and $\beta + \gamma < 0$.
\end{lemma}

\begin{proof}
  We use the orthogonality and self-adjointness of Littlewood-Paley blocks to
  get
  \begin{align*}
    D (f, g, h) = & \underset{j \sim k}{\underset{i}{\sum}} (\Delta_i f,
    \Delta_j g \Delta_k h) - \underset{k}{\underset{i \lesssim j}{\sum}}
    (\Delta_i f \Delta_j g, \Delta_k h)\\
    = & \underset{j \sim k}{\underset{i}{\sum}} (\Delta_i f \Delta_j g,
    \Delta_k h) - \underset{k}{\underset{i \lesssim j}{\sum}} (\Delta_i f
    \Delta_j g, \Delta_k h)\\
    = & \underset{j \sim k}{\underset{i}{\sum}} (\Delta_i f \Delta_j g,
    \Delta_k h) - \underset{k \sim j}{\underset{i \lesssim j}{\sum}} (\Delta_i
    f \Delta_j g, \Delta_k h)\\
    = & \underset{i \gtrsim j \sim k}{\sum} (\Delta_i f \Delta_j g, \Delta_k
    h)\\
    \leqslant & \underset{i \gtrsim j \sim k}{\sum} \| \Delta_i f \|_{L^{p_1}}
    \| \Delta_j g \|_{L^{p_2}} \| \Delta_k h \|_{L^{p_3}}
  \end{align*}
  by H{\"o}lder's inequality. Now using $\alpha + \beta + \gamma = 0$ and
  $\alpha > 0 \text{ and } \beta + \gamma < 0$ we have
  \begin{align*}
    D (f, g, h) \leqslant & \underset{i \gtrsim j \sim k}{\sum} 2^{i (\alpha +
    \beta + \gamma)} \| \Delta_i f \|_{L^{p_1}} \| \Delta_j g \|_{L^{p_2}} \|
    \Delta_k h \|_{L^{p_3}}\\
    \lesssim & \left\| 2^{i \alpha} \| \Delta_i f \|_{L^{p_1}}
    \right\|_{l_i^{\infty}} \underset{i \gtrsim j \sim k}{\sum} 2^{i (\beta +
    \gamma)} \| \Delta_j g \|_{L^{p_2}} \| \Delta_k h \|_{L^{p_3}}\\
    \lesssim & \| f \|_{{{B^{\alpha}_{p_1 \infty}} } } \underset{j \sim
    k}{\sum} 2^{j (\beta + \gamma)} \| \Delta_j g \|_{L^{p_2}} \| \Delta_k h
    \|_{L^{p_3}}\\
    \lesssim & \| f \|_{{{B^{\alpha}_{p_1 \infty}} } } \underset{j \sim
    k}{\sum} 2^{j \beta} \| \Delta_j g \|_{L^{p_2}} 2^{k \gamma} \| \Delta_k h
    \|_{L^{p_3}}\\
    \lesssim & \| f \|_{{{B^{\alpha}_{p_1 \infty}} } } \| g
    \|_{{{B^{\beta}_{p_2 2}} } } \| h \|_{{{B^{\gamma}_{p_3 2}} } }
  \end{align*}
  having bounded $\underset{i \gtrsim j}{\sum} 2^{- i \sigma} \sim 2^{- j
  \sigma}  \text{for } \sigma > 0$ to remove the sum in $i.$
  
  This concludes the proof.
\end{proof}

Next we have a similar bound in spirit involving iterated para- and resonant
products. We define the \textit{triple resonant product}
\begin{equation}
  \Pi (g, z, h) \assign \underset{j \sim k \sim l}{\sum} \Delta_j g \Delta_k z
  \Delta_l h.
\end{equation}
\begin{lemma}
  The map
  \[ D^{(2)} (f, g, h, z) \assign (f, \Pi (g, z, h)) - ((f \prec g) \prec z,
     h) \]
  defined for smooth functions, extends to a bounded map
  \[ B_{p_1 \infty}^{\alpha} \times B_{p_2 \infty}^{\beta} \times B_{p_3
     2}^{\gamma} \times B_{p_4 2}^{\delta} \rightarrow \mathbb{R} \]
  for $1 = \frac{1}{p_1} + \frac{1}{p_2} + \frac{1}{p_3} + \frac{1}{p_4}$ and
  $\alpha + \gamma + \delta \geqslant 0, \beta + \gamma + \delta > 0$ and
  $\gamma + \delta < 0$
\end{lemma}

\begin{proof}
  We have (using the product property $\mathrm{supp} \widehat{\Delta_i \cdummy
  \Delta_j} \subset \underset{k \sim i + j}{\cup} \mathrm{supp}
  \widehat{\Delta_k}$ and the orthogonality $\Delta_j \Delta_k = 0$ for $| j -
  k | > 1 $ )
  \begin{eqnarray*}
    D^{(2)} (f, g, h, z) & = & \underset{j \sim k \sim l}{\underset{i}{\sum}}
    (\Delta_i f, \Delta_j g \Delta_k z \Delta_l h) - \underset{l}{\underset{j
    \lesssim k}{\underset{i \lesssim j}{\sum}}} ((\Delta_i f \Delta_j g)
    \Delta_k z, \Delta_l h)\\
    & = & \underset{j \sim k \sim l}{\underset{i}{\sum}} (\Delta_i f,
    \Delta_j g \Delta_k z \Delta_l h) - \underset{}{\underset{l \sim
    k}{\underset{j \sim k}{\underset{i \lesssim j}{\sum}}}} (\Delta_i f
    \Delta_j g, \Delta_k z \Delta_l h)\\
    & = & \underset{j \sim k \sim l}{\underset{i}{\sum}} (\Delta_i f,
    \Delta_j g \Delta_k z \Delta_l h) - \underset{}{\underset{k \sim
    l}{\underset{j \sim k}{\underset{i \lesssim j}{\sum}}}} (\Delta_i f,
    \Delta_j g \Delta_k z \Delta_l h)\\
    & = & \underset{j \sim k \sim l}{\underset{i \gtrsim j}{\sum}} (\Delta_i
    f, \Delta_j g \Delta_k z \Delta_l h)\\
    & \leqslant & \underset{i \gtrsim j \sim k \sim l}{\sum} \| \Delta_i f
    \|_{L^{p_1}} \| \Delta_j g \|_{L^{p_2}} \| \Delta_k z \|_{L^{p_3}} \|
    \Delta_l h \|_{L^{p_4}}
  \end{eqnarray*}
  using H{\"o}lder's inequality in the final step. Then we bound, using
  $\alpha + \gamma + \delta \geqslant 0, \beta \geqslant 0$ and $\gamma +
  \delta < 0$
  \begin{eqnarray*}
    D^{(2)} (f, g, h, z) & \leqslant & \underset{i \gtrsim j \sim k \sim
    l}{\sum} 2^{i (\alpha + \gamma + \delta)} \| \Delta_i f \|_{L^{p_1}} 2^{j
    \beta} \| \Delta_j g \|_{L^{p_2}} \| \Delta_k z \|_{L^{p_3}} \| \Delta_l h
    \|_{L^{p_4}}\\
    & \lesssim & \left\| 2^{i \alpha} \| \Delta_i f \|_{L^{p_1}}
    \right\|_{l_i^{\infty}} \underset{i \gtrsim j \sim k \sim l}{\sum} 2^{i
    (\gamma + \delta)} 2^{j \beta} \| \Delta_j g \|_{L^{p_2}} \| \Delta_k z
    \|_{L^{p_3}} \| \Delta_l h \|_{L^{p_4}}\\
    & \lesssim & \| f \|_{{{B^{\alpha}_{p_1 \infty}} } } \underset{j \sim k
    \sim l}{\sum} 2^{j (\beta + \gamma + \delta)} \| \Delta_j g \|_{L^{p_2}}
    \| \Delta_k z \|_{L^{p_3}} \| \Delta_l h \|_{L^{p_4}}\\
    & \lesssim & \| f \|_{{{B^{\alpha}_{p_1 \infty}} } } \underset{j \sim k
    \sim l}{\sum} 2^{j \beta} \| \Delta_j g \|_{L^{p_2}} 2^{k \gamma} \|
    \Delta_k z \|_{L^{p_3}} 2^{l \delta} \| \Delta_l h \|_{L^{p_4}}\\
    & \lesssim & \| f \|_{{{B^{\alpha}_{p_1 \infty}} } } \| g
    \|_{{{B^{\beta}_{p_2 \infty}} } } \| z \|_{{{B^{\gamma}_{p_3 2}} } } \| h
    \|_{{{B^{\delta}_{p_4 2}} } }
  \end{eqnarray*}
  having used the geometric series $\underset{i \gtrsim j}{\sum} 2^{- i
  \sigma} \sim 2^{- j \sigma} \text{ for } \sigma > 0$ to get rid of the sum
  in $i$ and Hölder's inequality for the other sums.
\end{proof}

This result is different from the previous one, since the first lemma has a
resonant product that is ill-defined by itself but by duality can be defined.
In this second result, it is rather the converse, that an ill-defined dual
pairing can be defined by using that one term has a paraproduct structure and
then defining it by duality.

\begin{corollary}
  \label{cor:adj2}We may bound
  \begin{equation}
    | ((f \prec g) \prec z, h) | \lesssim \| f \|_{\mathcal{C}^{\alpha}} \| g
    \|_{\mathcal{C}^{\beta}} \| z \|_{\mathcal{H}^{\gamma} } \| h
    \|_{\mathcal{H}^{\delta}}
  \end{equation}
  for $\alpha + \gamma + \delta > 0, \beta + \gamma + \delta > 0$ and $\gamma
  + \delta < 0$. 
\end{corollary}

\end{document}